\documentclass[10pt]{article}
\usepackage[top=3.5cm, bottom=3cm, left=3cm, right=2.5cm]{geometry} 

\usepackage{amssymb,amsmath}
\usepackage{amsthm}
\usepackage{amsfonts}
\usepackage{amscd}

\usepackage{subfigure}
\usepackage{graphicx}
\usepackage{epsfig}

\usepackage{esvect} 
\usepackage{enumerate}
\usepackage{enumitem} 

\usepackage{dsfont}
\usepackage{mathtools}
\usepackage{hyperref}

\usepackage{pifont}
\newcommand{\cmark}{\ding{51}}%
\newcommand{\xmark}{\ding{55}}%
\newcommand{\N}{\mathbb{N}}
\newcommand{\R}{\mathbb{R}}

\newcommand{\bc}{{\boldsymbol{c}}}
\newcommand{\bd}{{\boldsymbol{d}}}
\newcommand{\bw}{{\boldsymbol{w}}}

\theoremstyle{plain}
\newtheorem{thm}{Theorem}[section]
\theoremstyle{plain}
\newtheorem{lem}[thm]{Lemma}
\newtheorem{prop}[thm]{Proposition}

\theoremstyle{definition}
\newtheorem{defi}[thm]{Definition}
\newtheorem{rem}[thm]{Remark}
\newtheorem{assumption}[thm]{Assumption}

\newtheorem{ex}[thm]{Example}

\newcommand{\ga}{\alpha}
\newcommand{\gb}{\beta}
\newcommand{\gd}{\delta}
\newcommand{\eps}{\ensuremath{\varepsilon}}

\renewcommand{\gg}{\gamma}
\newcommand{\gk}{\kappa}
\newcommand{\gl}{\lambda}

\newcommand{\gs}{\sigma}

\newcommand{\gL}{\Lambda}

\newcommand{\cA}{\mathcal{A}}
\newcommand{\cB}{\mathcal{B}}
\newcommand{\cC}{\mathcal{C}}
\newcommand{\cD}{\mathcal{D}}
\newcommand{\cE}{\mathcal{E}}

\newcommand{\cG}{\mathcal{G}}
\newcommand{\cH}{\mathcal{H}}
\newcommand{\cI}{\mathcal{I}}

\newcommand{\cK}{\mathcal{K}}
\newcommand{\cL}{\mathcal{L}}

\newcommand{\cN}{\mathcal{N}}
\newcommand{\cO}{\mathcal{O}}
\newcommand{\cP}{\mathcal{P}}

\newcommand{\cR}{\mathcal{R}}
\newcommand{\cS}{\mathcal{S}}

\newcommand{\cX}{\mathcal{X}}
\newcommand{\cY}{\mathcal{Y}}
\newcommand{\cZ}{\mathcal{Z}}

\newcommand{\bC}{\mathbb{C}}

\newcommand{\bN}{\mathbb{N}}

\newcommand{\bR}{\mathbb{R}}

\newcommand{\bT}{\mathbb{T}}

\newcommand{\bZ}{\mathbb{Z}}

\newcommand{\rC}{{\rm C}}

\newcommand{\mfA}{\mathfrak{A}}

\newcommand{\mfD}{\mathfrak{D}}

\newcommand{\mfI}{\mathfrak{I}}

\newcommand{\mfg}{\mathfrak{g}}
\newcommand{\mfu}{\mathfrak{u}}

\newcommand*{\lrscript}[5]{{\vphantom{#1}}_{#2}^{#3}{#1}_{#4}^{#5}}
\newcommand{\dualpair}[4]{\ensuremath{\lrscript{\langle}{#1}{}{}{} #3 ,#4 \rangle_{#2}}}
\newcommand{\be}{\begin {equation}}
\newcommand{\ee}{\end  {equation}}
\newcommand{\bee}{\begin {equation*}}
\newcommand{\eee}{\end {equation*}}
\newcommand{\ol}{\overline}

\newcommand{\floor}[1]{\left\lfloor #1 \right\rfloor}
\newcommand{\ceil}[1]{\left\lceil #1 \right\rceil}

\newcommand{\abs}[1]{\left\vert#1\right\vert}

\newcommand{\indi}{\mathds{1}}

\DeclareMathOperator*{\esssup}{\text{ess\,sup}}

\newcommand{\KL}{{Karhunen-Lo\`{e}ve }}

\newcommand{\Lip}{{\rm Lip}}

\newcommand{\inp}[2]{\langle #1, #2 \rangle}
\newcommand{\inpc}[2]{\left\langle #1, #2\right\rangle}
\newcommand{\norm}[1]{\left\|#1\right\|}
\newcommand{\set}[1]{\left\{#1\right\}}

\let\inf\relax \DeclareMathOperator*\inf{\vphantom{p}inf}

\usepackage[normalem]{ulem}
\usepackage{xcolor}
\usepackage{todonotes}

\begin{document}
	
	\title{Deep Operator Network Approximation Rates 
                            \\ for Lipschitz Operators}

	\date{\today}
	
	\author{
		Christoph Schwab\thanks{ChS acknowledges a visit to the 
                Isaac Newton Institute for Mathematical Sciences, Cambridge, UK,
                for support and hospitality during 16-28 April, 2023 in the 
                research semester ``Data Driven Engineering (DDE)'',
                supported by EPSRC Grant Number EP/R014604/1.
                Stimulating discussions on operator learning during workshop DDEW3 are warmly acknowledged.}
       	\footnote{ETH Zürich, Seminar for Applied Mathematics, schwab@math.ethz.ch}
		\and
		Andreas Stein 
		\footnote{ETH Zürich, Seminar for Applied Mathematics, andreas.stein@sam.math.ethz.ch, 
                Partially supported by the ETH FoDS (Foundations of Data Science) Initiative}
		\and 
		Jakob Zech \footnote{Heidelberg University, Interdisciplinary Center for Scientific Computing (IWR), jakob.zech@uni-heidelberg.de}
	}
	
	\maketitle
	
\begin{abstract}
	We establish universality and
        expression rate bounds for a class of neural Deep Operator Networks (DON)
	emulating Lipschitz (or H\"older) continuous maps $\cG:\cX\to\cY$
	between (subsets of) separable Hilbert spaces $\cX$, $\cY$.
	The DON architecture considered uses linear encoders $\cE$ and decoders $\cD$ 
	via (biorthogonal) Riesz bases of $\cX$, $\cY$, and an approximator network
	of an infinite-dimensional, parametric coordinate map 
	that is Lipschitz continuous on the sequence space $\ell^2(\bN)$. 
	Unlike previous works \cite{herrmann2022neural,MS21_984}
	which required for example $\cG$ to be holomorphic,
	the present expression rate results
        require mere Lipschitz (or H\"older) continuity of $\cG$.
	Key in the proof of the present expression rate bounds  
        is the use of either super-expressive activations (e.g.\
	\cite{yarotsky2021elementary,shen2021neural} and the references there)
	which are inspired by the Kolmogorov superposition theorem
	(\cite{KolmogAN} or \cite[Chap.11]{LorentzGG} for a comprehensive exposition), 
        or of nonstandard NN architectures with
	standard (ReLU) activations as recently proposed in \cite{shen2022neural}.
	We illustrate the abstract results by approximation rate bounds for 
        emulation of 
        a) solution operators for parametric elliptic variational inequalities, 
        and 
        b) Lipschitz maps of Hilbert-Schmidt operators.
      \end{abstract}

      {\bf keywords:} Neural Networks, Operator Learning,
      Curse of Dimensionality, Lipschitz Continuous Operators
      
      {\bf subject classifiction:} 41A65, 68T15, 68Q32
      
%
\section{Introduction}
\label{sec:intro}
Following widespread use in data and image classification and
forecasting, recent years have seen 
development of 
\emph{Deep Neural Networks} (DNNs for short) 
in scientific computing as universal, 
and highly versatile approximation architectures, that
challenge established numerical approximations such as Finite Element
and Finite Difference discretizations of partial differential
equations (PDEs) and integral equations.  
Algorithmic approaches are
based on inserting DNNs into suitable (variational, weak, least
squares, strong etc.)  forms of the PDE under consideration.  Besides
the universality of DNNs which was mathematically established early on
(see, e.g. \cite{PinkusActa} and references there), in recent years a
much more detailed picture on the expressive power of DNNs has
emerged.  In particular, feed-forward DNNs with suitable architectures
and activation functions can emulate practically all standard spline-
and Finite Element approximation spaces commonly used in the numerical
analysis of PDEs (see, e.g., \cite{LODS22_2959}) and, in particular,
also high-order FE spaces \cite{OPS20_2738} with corresponding
spectral or exponential approximation rates \cite{OSZ22_2771}. 
We refer to \cite{PiNNCFDRev} and the references there for some of the
algorithmic developments.  A particular feature of numerical
approximations of PDE solutions based on DNNs as approximation
architectures that was observed in practice was the apparent
insensitivity of the DNN approximation quality to the so-called 
``curse of dimensionality'' (CoD for short). 
This is particularly relevant for approximating maps
\begin{equation}\label{eq:cG}
  \cG:\cX\to\cY
\end{equation}
between (in general, infinite-dimensional) separable Hilbert spaces\footnote{ 
More generally, separably-valued maps $\cG$ 
into an otherwise nonseparable target space $\cY$ may be considered.
In \cite[{Section~9, App.~B}]{NeurOp} additional conditions on separable 
Banach spaces $\cX$ and $\cY$ necessary to extend the present arguments
to this more general setting are discussed.
}
$\cX$ and $\cY$. 
Operators $\cG$ as in \eqref{eq:cG} emerge for example as
parameter-to-solution mappings for parametric PDEs within the field of
Uncertainty Quantification 
(see, e.g., \cite{molinaro2023neural} and the references there),
or in so-called digital twins of complex, physical systems governed
by partial differential equations (PDEs) 
(see \cite{kapteyn2021probabilistic} and the references there).
Owing to the infinite dimension of $\cX$ and $\cY$ in \eqref{eq:cG},
    efficient numerical approximations of maps $\cG$ are to overcome the CoD.

Several (intrinsically different) mechanisms for overcoming the 
CoD in DNN emulations have been identified and mathematically justified recently. 
This includes the seminal work of A.\ \emph{Barron}
\cite{barron1993universal}, \emph{Monte-Carlo path simulation} type arguments
(e.g.\ \cite{GS21_2893,HJKTvW} and the references there), and the
emulation of sparse (generalized) \emph{polynomial chaos expansions} 
(e.g.\ \cite{BCDS17_2452,DNSZ22_2957}) by DNNs (e.g.\ \cite{ChSJZ,MR4409717,ChSJZ2}).

Specifically, in \cite{ChSJZ,MR4409717,ChSJZ2}, 
a parametric representation of inputs $x\in\cX$ of $\cG$ 
was used to prove DNN emulation rates for approximating $\cG$.  
The construction used DNNs whose depth
scales polylogarithmic in the parameter dimension, and polynomially in
the DNN expression accuracy (i.e., emulation fidelity).  Key in the
proofs of these results is the \emph{holomorphic} dependence of
$\cG(x)$ on the input $x$. 
The related DNN emulation results
were obtained with sparsely connected, deep feedforward NNs with ReLU
or smooth (e.g.\ sigmoidal or $\tanh(\cdot)$) activation.
DNN emulation rate results that are free from the CoD for 
\emph{low regularity maps} $\cG$ between function spaces were obtained
e.g. using the so-called Feynman-Kac representation of
solutions of Kolmogorov PDEs in (jump-)diffusion models.
These results used ReLU
DNNs of moderate depth \cite{GS21_2893,HJKTvW},
but the error bounds hold in a mean-square sense or only with high probability.

While quantified, 
parametric holomorphy of solution families of
parametric PDEs has been verified in many settings 
(particularly in elliptic and parabolic PDEs, e.g.\
\cite{HS13_409,ZDS19_2621,JSZ17_2339,CSZ18_2319,HS21_2779}),
there are broad classes of applications where relevant maps are
H\"older or Lipschitz, but not holomorphic. 
One purpose of the present paper is to obtain mean-square
DNN expression rate bounds for \emph{Operator Network} (ONet)
emulations with architecture \eqref{eq:DONarch} below, 
of Lipschitz (and, more generally, H\"older smooth) maps $\cG$ 
between separable Hilbert spaces.
\subsection{Previous work for operator networks}
\label{sec:Dons}
A rather recent line of research uses so-called 
\emph{Operator Networks} to emulate the possibly nonlinear
input-output map $\cG$, 
such as for example the coefficient-to-solution map in linear,
elliptic divergence form PDEs of second order.
A variety of DNN architectures has been put forward recently with the
aim of efficient operator emulation, with distinct architectures
tailored to the emulation of particular operators.  
A number of acronyms labelling these DNN classes has been coined
(``deepONets'' \cite{deepONet}, Fourier Neural Operators ``FNOs''
\cite{NeurOp,FNOGeo}, 
UNet architectures combined with FNOs ``U-FNOs'' \cite{UFNo}, 
encoders based on transformers, 
etc.). 
We refer to
\cite{NeurOp,LMK22,PiDON,PCOGDerInfNOp22,li2022transformer,cao2021choose}
and the references there.

In this paper, we discuss an architecture that belongs to the same
general category as those proposed in, for instance,
\cite{HESTHAVEN201855,deepONet,herrmann2022neural}. 
It reduces the task of
approximating $\cG$ to that of emulating (components of)
\emph{countably-parametric maps $G: \ell^2(\bN)\to\ell^2(\bN)$} 
with DNNs: using an appropriate 
\emph{encoder} $\cE_\cX:\cX\to\ell^2(\N)$ and
\emph{decoder} $\cD_{\cY}:\ell^2(\N)\to\cY$, 
the map $\cG$ in \eqref{eq:cG} allows the structural representation
%
\begin{equation}\label{eq:Gstrct}
  \cG = \cD_\cY \circ G \circ \cE_{\cX}.
\end{equation}
For the deepONet ``branch-trunk'' architecture \cite{ChenChen1993,deepONet}, 
expression rate bounds were first investigated in \cite{LMK22} for
holomorphic maps $\cG$, with inputs $x\in\cX$ stemming from a
Karhunen-Lo\'eve expansion with exponentially decaying
eigenvalues. 
In \cite{herrmann2022neural}, we considered the approximation of
holomorphic maps using encoders and decoders based on frame
representations of the input and output in $\cX$ and $\cY$. We
showed that dimension-independent algebraic expression rates can be
obtained depending on the smoothness of the spaces $\cX$, $\cY$, for
example in terms of Sobolev or Besov regularity.  In this situation,
the frame coefficients typically exhibit only polynomial decay
rather than exponential decay.  
For Lipschitz continuous
$\mathcal{G}$, the recent preprint \cite{lanthaler2023operator} offers lower
bounds, indicating that the PCA-net using PCA-based encoders and decoders with 
standard, deep feedforward neural network based approximators 
generally can \emph{not} avoid the CoD and is unable to achieve algebraic convergence.
%
\subsection{Contributions}
\label{sec:Contr}
%
We consider (nonlinear) maps $\cG:\cX\to\cY$
which are Lipschitz or H\"older regular between
infinite-dimensional, separable Hilbert spaces. 
Endowing $\cX$ and $\cY$ with stable (Riesz) 
bases and corresponding encoder/decoder
pairs, such maps admit the structure \eqref{eq:Gstrct}. 
We prove algebraic expression rate bounds for corresponding
finite-parametric DNN surrogates,
which are a key step in the mathematical analysis of DNN operator surrogates 
as outlined e.g. in \cite[Section~2]{NeurOp}.
The considered 
linear encoders and decoders built on Riesz-bases in $\cX$ and $\cY$
accommodate a variety of currently used Operator Nets, 
comprising Fourier- and KL expansions. 
If $\cX$ and $\cY$ are function
spaces over domains $D\subset \R^d$ and $D'\subset \R^{d'}$ 
(possibly of different dimensions $d$ and $d'$; for 
simplicity of notation, we constrain here to $D=D'$ and $d=d'$),
exhibiting sufficient smoothness to allow
continuous embeddings into the space of continuous functions on $\overline{D}$,
also point-collocation as e.g. in the deepONet architecture
\cite{deepONet} is in principle admitted. 
Furthermore, we present a
universal approximation theorem for continuous (but not necessarily
H\"older continuous) operators $\cG:\cX\to\cY$, which guarantees
uniform convergence on any compact subset of $\cX$.

Our main result, Theorem~\ref{thm:MainRes}, states that there exist
\emph{finite-parametric approximator surrogates $\widetilde G$} of the
countably-parametric maps $G$ in \eqref{eq:Gstrct} which, upon
insertion into \eqref{eq:Gstrct}, result in finite-parametric neural
operators $\widetilde{\cG}$
offering (on ``higher-regularity'' inputs from subspaces of $\cX$)
algebraic consistency orders with $\cG$ that are free of the CoD.
This is achieved by leveraging recent progress in the construction of
DNN approximations for scalar functions of many variables of low
regularity.  These networks are either of standard feedforward
architecture, but leverage superior expressive power of DNNs by
invoking \emph{nonstandard activations} (termed ``super-expressive''
e.g.\ in \cite{shen2021neural,yarotsky2021elementary} inspired by the
Kolmogorov-Arnold superposition theorem, e.g.\
\cite{KolmogAN,LorentzGG} and the references there) or they are based
on \emph{nonstandard architectures} allowing a higher degree of
connectivity than DNNs of plain feedforward type (e.g.\
\cite{shen2022neural}).  As such, our results do not contradict the
recent lower bounds for ReLU-activated, feedforward DONs stated
e.g. in \cite{LMK22,lanthaler2023operator}.  We remark that
computational adaptation of finite-parametric activation functions in,
e.g., so-called ``PiNNs'', has been observed to be computationally
effective in applications in \cite{AdaptActivat}.  By resorting to
such nonstandard DNN architectures and activations in the emulation of
the (component functions of) map $G$ in \eqref{eq:Gstrct}, we identify
sufficient conditions in order for the maps $\cG$ to be emulated by
DNNs with accuracy that behaves algebraically in terms of the number
of neurons.  We finally observe that the present setting
\eqref{eq:Gstrct} is a particular case of a number of other DON
architectures (e.g. \cite{NeurOp,LMK22}).  The DONs analyzed can also
be viewed as a building block of the recently featured ``Nonlocal
Neural Operators'' (NNOs) in \cite{lanthaler2023nonlocal}.
\subsection{Notation}
\label{sec:notat}
Throughout, and unless explicitly stated otherwise, $\cX$ and $\cY$ shall denote
separable Hilbert spaces of infinite dimension.
	For any $p\in[1,\infty]$, let $\ell^p(\bN)$ denote the space of all 
        $p$-summable, real-valued sequences over $\bN$.
	The Borel $\gs$-algebra of any metric space $({\cZ}, d_{{\cZ}})$ 
        is generated by the open sets in ${\cZ}$ and denoted by $\cB({\cZ})$.
	For any $\gs$-finite and complete measure space $(E,\cE,\mu)$, 
	a Banach space $(\cZ, \left\|\cdot\right\|_{\cZ})$, 
	and 
	summability exponent $p\in[1,\infty]$,
	we define the Lebesgue-Bochner spaces 
\begin{equation*}
  L^p(E, \mu; \cZ) := \{\varphi:E\to \cZ:\;
  \text{$\varphi$ is strongly measurable and $\|\varphi\|_{L^p(E, \mu; \cZ)}<\infty$}  \},
\end{equation*}
where
\begin{equation*}
  \|\varphi\|_{L^p(E, \mu; \cZ)}
   :=
  \begin{cases}
    \left(\int_{E}\|\varphi(x)\|_{\cZ}^p\mu(dx)\right)^{1/p},\quad &p\in[1,\infty) 
    \\
    \esssup\limits_{x\in E} \|\varphi(x)\|_{\cZ},\quad &p=\infty.
  \end{cases}
\end{equation*}
In case that $\cZ=\bR$, we use the shorthand notation
$L^p(E,\mu):=L^p(E,\mu;\bR)$.  
In case that $E\subseteq\bR^d$ is a subset of
Euclidean space, we assume $\cE=\cB(E)$ and $\mu$ is the Lebesgue
measure, and write $L^p(E):=L^p(E,\mu;\bR)$, unless stated otherwise.
We further denote by $\norm{\cdot}_2$ the Euclidean norm on
$E\subseteq\bR^d$.
\subsection{Layout}
\label{sec:layout}
This paper is organized as follows: In Section \ref{sec:setting}
  we briefly present the setting of our consistency analysis, which
  consists of linear $N$-term encoder/decoder pairs in the domain and
  the range of the operator under consideration.  Section
  \ref{sec:univers} provides a universal approximation theorem valid
  for the approximation of continuous maps $\cG:\cX\to\cY$. To show
  convergence rates, in Section \ref{sec:dim_trunc} we first discuss
  preliminary results regarding the dimension truncation of the
  encoded input and output. We use these results in Section
  \ref{sec:nn_approx} to give our main result, which shows algebraic
  convergence rates for the approximation of Lipschitz-continuous
  operators. 
	We comment on the extension to e.g. 
	H\"older continuity in Section \ref{sec:Extns}. 

  To illustrate the scope of our abstract results, in Section
  \ref{sec:Examples} we prove ONet expression rate bounds for 
  particular classes of Lipschitz mappings $\cG$ covered by our setting:
  in Section~\ref{sec:VI}, we consider 
  solution maps to parametric elliptic variational inequalities as arise e.g. 
  in optimal stopping, optimal control, and in contact problems in mechanics. 
  In Section~\ref{sec:LipHS}, we obtain ONet expression rate bounds for 
  Lipschitz maps of Hilbert-Schmidt operators acting on separable Hilbert spaces.

	\section{Setting}
	\label{sec:setting}
	%
	We adopt the setting from~\cite{herrmann2022neural}.
	Let $(\cX,\inp{\cdot}{\cdot}_\cX )  $ and $(\cY, \inp{\cdot}{\cdot}_\cY)$ 
        denote two separable Hilbert spaces over $\bR$. 
        We consider neural network emulations of 
        infinite-dimensional (non-linear) operators $\cG:\cX\to \cY$. 
	To this end, we rewrite $\cG$ in the form $\cG = \cD_Y\circ G\circ \cE_{\cX}$, 
        where $\cE_\cX:\cX\to \ell^2(\bN)$ and $\cD_\cY:\ell^2(\bN)\to \cY$ 
        are linear \emph{encoder} and \emph{decoder}, respectively,
        and   
        $G:\ell^2(\bN)\to \ell^2(\bN)$ is an infinite-parametric map.
	We employ linear encoders or ``analysis operators'' that 
        convert function space inputs from $\cX$ 
        with suitable representation systems to coefficient sequences such as
        e.g.\ Fourier coefficients w.r.t. a fixed orthonormal basis of $\cX$ 
        (such as, e.g., principal component
         representations with respect to a \KL (KL) basis corresponding to
         a covariance operator of a probability measure on $\cX$, leading
         to the so-called ``PCA-ONet'', see e.g. \cite{lanthaler2023operator}),
        while our decoders perform the converse ``synthesis'' operation w.r.t. 
        to another fixed representation system in $\cY$. 
        We do not insist on orthogonal representation systems in $\cX$ or in $\cY$, 
        and rather admit Riesz bases of $\cX,\cY$ and their 
        analysis and synthesis operators as encoders and decoders.   
	We build finite-parametric deep operator network surrogates of the
        general architecture
	\begin{equation}\label{eq:DONarch}
		\widetilde\cG:=\cD_Y\circ \widetilde G\circ\cE_{\cX}, 
	\end{equation}
	Theorem~\ref{thm:Universality} implies in particular
        universality of the DON where the \emph{approximator network} 
        $\widetilde G:\ell^2(\bN)\to \ell^2(\bN)$ 
        is a neural network emulation of the infinite-parametric, Lipschitz-continuous 
        maps $G$, 
        that allow for an efficient approximation of $\cG$ 
        on suitable sets $S\subseteq\cX$.
	More precisely, we consider convergence of the mean-squared error 
	\begin{equation}
		\left(\int_S \|\cG(x) - \widetilde\cG(x)\|_\cY^2 \,\mu(dx)\right)^{1/2},
	\end{equation}
	where $\mu$ is an appropriate measure on $(S, \cB(S))$.
	\subsection{Encoders and Decoders}
	\label{sec:riesz_bases}
	%
	We admit \emph{linear en- and decoders} built from
        Riesz bases as representation systems in $\cX$ and $\cY$, 
        comprising in particular Fourier-, Wavelet- and KL-bases 
        for input and output parametrization. 
        The following notion of a Riesz basis is one of several equivalent
        definitions.
        It follows as a consequence of
        \cite[Definition 3.6.1]{christensen2003introduction} 
        and \cite[Theorem 3.6.6]{christensen2003introduction}.  
        We refer to \cite[Section 3.6]{christensen2003introduction} for details.
	\begin{defi}
          Let $(\cH, \inp{\cdot}{\cdot}_\cH)$ be a separable Hilbert space.  
          A complete sequence
          $\mathbf\Psi_\cH=(\psi_i, i\in\bN)\subset \cH$ is called
          \emph{Riesz basis of $\cH$} if there exists
          constants\footnote{The constants $\gl_{\cH}$ and
              $\gL_{\cH}$ depend on the Riesz basis
              $\mathbf{\Psi}_\cH$ but not on $\cG$. They quantify the sensitivity 
              of the approximator error $G - \tilde{G}$ on the output error.
              For conciseness, we do not
              explicitly indicate this dependence in our notation.}
          $0 < \gl_{{\cH}}\le \gL_{{\cH}}<\infty$ such that for
          all $\bc=(c_i,i\in\bN)\in\ell^2(\bN)$ there holds
          \begin{equation}\label{eq:norm_equiv}
            \gl_{{\cH}} \norm{\bc}_{\ell^2(\bN)}^2
            =
            \gl_{{\cH}} \sum_{i\in\bN} c_i^2
            \le 
            \norm{\sum_{i\in\bN} c_i\psi_i}_\cH^2
            \le 
            \gL_{{\cH}} \sum_{i\in\bN} c_i^2
            =
            \gL_{{\cH}} \norm{\bc}_{\ell^2(\bN)}^2.
          \end{equation} 
%
	\end{defi}
	
	For any Riesz basis $\mathbf\Psi_\cH=(\psi_i, i\in\bN)$, there
        exists another (unique for given
        $\mathbf\Psi_\cH=(\psi_i, i\in\bN)$) Riesz basis
        $\widetilde{\mathbf\Psi}_\cH=(\widetilde \psi_i, i\in\bN)$,
        called \emph{dual basis} or \emph{biorthogonal system} to
        $\mathbf\Psi_\cH$, such that there holds
	\begin{equation*}
          f=\sum_{i\in\bN}\inp{f}{\widetilde \psi_i}_\cH\psi_i \quad\text{for all $f\in\cH$\qquad and \qquad}
          \inp{\psi_i}{\widetilde\psi_j}_\cH=\gd_{ij}\quad\text{for all $i,j\in\bN$,}
	\end{equation*}
	see e.g.\ \cite[Theorem 3.6.2]{christensen2003introduction}.
	If $\mathbf\Psi_\cH$ is an orthonormal basis (ONB) of $\cH$,
        then $\widetilde{\mathbf\Psi}_\cH=\mathbf\Psi_\cH$ and
        $\gl_{\cH} = \gL_{\cH} = 1$.
	
	For the remainder of this article, we fix Riesz bases
        $\mathbf\Psi_\cX=(\psi_i, i\in\bN)\subset \cX$ and
        $\mathbf\Psi_\cY=(\eta_j, j\in\bN)\subset \cY$ for $\cX$ and
        $\cY$, respectively, and denote their corresponding dual bases
        by
        $\widetilde{\mathbf\Psi}_\cX=(\widetilde \psi_i,
        i\in\bN)\subset \cX$ and
        $\widetilde{\mathbf\Psi}_\cY=(\widetilde\eta_j,
        j\in\bN)\subset \cY$.
	The associated Riesz constants from~\eqref{eq:norm_equiv} are
        denoted by
$$
0 < \gl_\cX\le\gL_\cX<\infty\;, \quad \mbox{ and } \quad 0 <
\gl_\cY\le\gL_\cY < \infty \;.
$$
For given Riesz bases $\mathbf\Psi_\cX$ of $\cX$ and $\mathbf\Psi_\cY$
of $\cY$, we define the encoder/decoder pairs
\begin{equation}
\label{eq:Xencod}
  \cE_{\cX}:\cX\to \ell^2(\bN),\quad x\mapsto (\inp{x}{\widetilde \psi_i}_\cX, i\in\bN),
  \qquad \quad
  \cD_{\cX}:\ell^2(\bN)\to\cX,\quad \bc\mapsto \sum_{i\in\bN}c_i\psi_i,
\end{equation}
and
\begin{equation}
\label{eq:Yencod}
  \cE_{\cY}:\cY\to \ell^2(\bN),\quad y\mapsto (\inp{y}{\widetilde \eta_j}_\cY, j\in\bN),
  \qquad \quad 
  \cD_{\cY}:\ell^2(\bN)\to\cY,\quad \bc\mapsto \sum_{j\in\bN}c_j\eta_j.
\end{equation}
On the entire spaces, the encoders/decoders are boundedly invertible
mappings and it holds
$$         
\cD_\cH\circ\cE_\cH = I_\cH\quad\mbox{ for }\quad \cH\in\{\cX,\cY\}.
$$
We mention that explicit, ``Finite-Element-like'' \emph{constructions
  of piecewise polynomial biorthogonal systems in polytopal domains}
are available, see e.g. \cite{DKU99,CDFS13,DahmStevEBEWav1999,LODS22_2959}.
	\subsection{Smoothness Scales $\cX^s$, $\cY^t$}
	\label{sec:smoothness_scales}
	%
        Our convergence rate analysis will be obtained on (in general compact) 
        subsets $\cX^s\subset\cX$ and $\cY^t\subset \cY$ of inputs / output pairs which  
        admit extra regularity.
        We postulate that, in terms of the Riesz bases $\mathbf\Psi_\cX$ and $\mathbf\Psi_\cY$, 
        this regularity takes the form of weighted summability of the 
        corresponding sequences of expansion coefficients.
        To formalize this condition, we next define
        a scale of Hilbert spaces depending on a smoothness parameter
        characterizing this type of coefficient decay.
	Typical instances of such ``smoothness spaces'' are Sobolev and Besov spaces 
        with $p$-integrable weak derivatives (e.g. \cite{TriebelWavDom} and the references there),
	or the Cameron-Martin space of the covariance operator of a Gaussian measure on $\cX$ or $\cY$
        (see \cite{LPFrame09}).
	
	Let then $\bw=(w_i,i\in\bN)\subset (0,1]$ 
        be a non-increasing sequence of weights such that
        $\bw\in\ell^{1+\eps}(\bN)$ for all $\eps>0$. 
        The latter condition is sufficient to derive the truncation error rates 
        in Proposition~\ref{prop:input_trunc} for arbitrary small $\delta>0$. 
        With this sequence, following \cite[Section 2]{herrmann2022neural}, 
        for all $s$, $t\ge 0$ we introduce  
        Hilbert spaces $\cX^s\subset \cX$, $\cY^t\subset \cY$ via their norms
	\begin{equation}
		\norm{x}_{\cX^s}^2:=\sum_{i\in\bN} \inp{x}{\widetilde \psi_i}_\cX^2w_i^{-2s},\qquad
		\norm{y}_{\cY^t}^2:=\sum_{j\in\bN} \inp{y}{\widetilde \eta_j}_\cY^2w_j^{-2t}.
	\end{equation}
	In order to streamline the presentation,
        we utilize the same sequence $\bw$ of weights $w_j$ to characterize $\cX^s$ and $\cY^t$, 
        but all of our subsequent results remain 
        valid for 
        distinct weighting sequences $\bw_\cX, \bw_\cY \in\ell^{1+\eps}(\bN)$.
	We further note that $\cX^s=\{x\in\cX: \|x\|_{\cX^s}<\infty\}$
        equipped with the scalar product 
        $\inp{x}{x'}_{\cX^s}:=\sum_{i\in\bN} \inp{x}{\widetilde
          \psi_i}_\cX \inp{x'}{\widetilde \psi_i}_\cX w_i^{-2s}$
        is a separable Hilbert space with Riesz basis
        $(w_i^s\psi_i, i\in\bN)$, see e.g.\ 
        \cite[Lemma 2.9 and Remark 2.10]{herrmann2022neural}.
        \section{Universality}
        \label{sec:univers}
        A universality result for the approximation of functionals
        mapping from compact subsets of $C([a,b])$ or $L^p([a,b])$ to
        $\R$ using neural networks was already established in the pioneering work of
        Chen and Chen in 1993 \cite{ChenChen1993}. 
        More recently, a universal
        approximation theorem for Lipschitz continuous $\cG:\cX\to\cY$
        utilizing the PCA-net architecture was proven in \cite{lanthaler2023nonlocal}. 
        Universality in the infinite width limit 
        was shown in the $L^2(\cX,\mu)$
        sense for measures $\mu$ possessing finite fourth moments,
        with the KL basis of the covariance of $\mu$, and
        relaxed later to require only 
        finite second moments for $\mu$ and $\cG$ 
        being a $\mu$-measurable map \cite[Theorem 3.1]{lanthaler2023operator}.

        We start our present approximation rate analysis by 
        proving an universal approximation theorem 
        which is valid for any continuous $\cG:\cX\to\cY$,
        with uniform convergence on compact subsets of $\cX$. 
        To state the result, we introduce
        the set of admissible activation functions as in \cite{LESHNO1993861}:
         \begin{align*}
          \cA:=\big\{
          		&\text{$\sigma\in L_{\rm loc}^\infty(\R)$ is not polynomial and the closure } \\ 
          		&\qquad\text{of the points of discontinuity has Lebesgue measure $0$}\big\}.
        \end{align*}
        ONets $\widetilde\cG$ as in \eqref{eq:DONarch} with the (components of the) approximator 
        $\widetilde G$ being a feedforward $\sigma$-NN 
        with activation $\sigma\in\cA$ are universal.

        \begin{thm}\label{thm:Universality}
          Let $\cX$, $\cY$ be two separable Hilbert spaces, let
          $\cG:\cX\to\cY$ be continuous and let $\sigma\in\cA$.

          Then there exists a sequence of operator nets
          $\widetilde \cG_n:\cX\to\cY$, $n\in\N$, with architecture
          \eqref{eq:DONarch} such that
          \begin{equation*}
          \forall x \in \cX: \quad 
            \lim_{n\to\infty}\widetilde\cG_n(x)=\cG(x) \;.
          \end{equation*}
          The convergence is uniform on every compact subset of $\cX$.
        \end{thm}
        The proof is given in Appendix \ref{app:ProofUni}.
We remark that Theorem~\ref{thm:Universality} 
implies in particular universality of the 
DON architecture \eqref{eq:DONarch} for 
maps between the Hilbertian Sobolev spaces 
$H^s(D)$ and $H^{s'}(D')$ for $s,s' \in \R$.
\begin{rem}\label{rem:nonlocal}
	We observe that the expressions \eqref{eq:Xencod}, \eqref{eq:Yencod} 
	for the coefficient sequences $\bc$ are ``\emph{nonlocal}'' 
	in terms of the function space inputs $x$ and $y$.
	Non-locality was highlighted in \cite{lanthaler2023nonlocal} 
	as important prerequisite for universality of a number of DON architectures.
        In their analysis, the authors also examined the possibility
        of using neural network approximations for the encoder and
        decoder mappings when $\cX$ and $\cY$ are function spaces on
        bounded domains.  Although it is technically possible to
        emulate $\cD_\cY$ and $\cE_\cX$ using DNNs in our context, we
        opted to not elaborate on this, in order to not overload the
        presentation.
\end{rem}

	\section{Dimension Truncation}
	\label{sec:dim_trunc}
	%
	Taking the cue from \cite{herrmann2022neural}, 
        within the architecture \eqref{eq:Gstrct}, 
        we approximate $\cG$ on finite-parametric subspaces of $\cX$.
        This implies that encoding will contain a form of ``dimension-truncation'' 
        where only a finite number $N$ of parameters in otherwise infinite-parametric,
        equivalent representations of inputs from $\cX$ are retained. 
        Importantly, due to the index-dependent relative importance of the 
        component maps $(g_j, j\in\bN)$ of the parametric map $G$ in \eqref{eq:Gstrct}.
        In Section~\ref{sec:SmthCl} we introduce the sets of admissible inputs
        for the ensuing expression rate analysis. We characterize in particular
        higher regularity of inputs and outputs in terms of weighted summability 
        of the coefficient sequences resulting from encoding in $\cX$ and $\cY$.
        \subsection{Smoothness Classes} 
        \label{sec:SmthCl}
	Let $r>0$, $s>\frac{1}{2}$ and define the ``cubes''
	\begin{equation*}
		C_r^s(\cX)
                :=
                \set{x\in\cX: \cE_{\cX}(x)\in \bigtimes_{i\in\bN}[-rw_i^s, rw_i^s]}
		=
		\set{x\in\cX: \sup_{i\in\bN} \abs{\inp{x}{\widetilde\psi_i}_\cX w_i^{-s}}\le r}.
	\end{equation*} 
	Note that for any $s'\in[0, s-\frac{1}{2})$ 
        there holds $C_r^s(\cX)\subset \cX^{s'}$ by \cite[Remark 3.2]{herrmann2022neural}. 
        Further, let 
	\begin{equation*}
		B_r(\cX^s):=\set{x\in\cX^s: \|x\|_{\cX^s}\le r}
	\end{equation*}
	denote the closed ball with radius $r>0$ in $\cX^s$. 
	Then, for any $\eps>0$ and
        with $r_\eps^2:=r^2 \sum_{i\in\bN} w_i^{1+2\eps} \in(0,\infty)$ 
        it holds
	\begin{equation}\label{eq:ball-cube-relation}
		B_r(\cX^s)\subseteq C_r^s(\cX) \subseteq B_{r_\eps}(\cX^{s-\frac{1}{2}-\eps}).
	\end{equation}
	Hence, error bounds on approximations of $\cG$ on $C_r^s(\cX)$ 
        trivially imply bounds on the closed ball $B_r(\cX^s)$.
	
	We fix the following assumption to derive finite-dimensional surrogates for $\cG$. 
	\begin{assumption}\label{ass:lipschitz}
		There exist $s>\frac{1}{2}$, $t$, $r>0$ and a constant $L_\cG>0$ such that 
                $\cG[C_r^s(\cX)]\in \cY^t$ and
		\begin{equation}\label{eq:lipschitz}
			\|\cG(x)-\cG(x')\|_{\cY^t}\le L_\cG \|x-x'\|_\cX, \quad x,x'\in C_r^s(\cX).
		\end{equation}
	\end{assumption}

	\begin{rem}
		The global Lipschitz continuity in Assumption~\ref{ass:lipschitz} 
                implies the linear growth bound
		\begin{equation}\label{eq:lin_growth}
			\|\cG(x)\|_{\cY^t}\le C(1+\|x\|_{\cX}),\quad x\in C_r^s(\cX),
		\quad \mbox{with}\quad  C:=\max(L_\cG, \norm{\cG(0)}_{\cY^t}).
		\end{equation}
	\end{rem}
	\subsection{Decoding Dimension Truncation}
	\label{sec:output_trunc}
	%
	For $N\in\bN$ define the \emph{restriction operator} 
        $\cR_N$ for sequences in $\ell^2(\bN)$ 
        via
	\begin{equation*}
		\cR_N:\ell^2(\bN)\to \ell^2(\bN), \quad 
                \bc=(c_i,i\in\bN) \mapsto (c_1,\dots, c_N, 0,0,\dots).
	\end{equation*} 
	We define the \emph{$N$-term output-truncated} approximation 
        $\cG_N:\cX\to\cY$ of $\cG$ 
        via 
	\begin{equation}\label{eq:NtermTrnc}
		\cG_N:=\cD_\cY\circ \cR_N \circ\cE_\cY\circ \cG.
	\end{equation} 
	
	\begin{prop}\label{prop:output_trunc}
		Under Assumption~\ref{ass:lipschitz}
                there exists a constant $C>0$ such that for 
                all $N\in\bN$ 
		\begin{equation*}
			\sup_{x\in C_r^s(\cX)} \norm{\cG(x)-\cG_N(x)}_\cY \le C w_{N+1}^t.
		\end{equation*}
	\end{prop}

	\begin{proof}
		For any $x\in C_r^s(\cX)$ and $\gd>0$ it holds by Assumption~\ref{ass:lipschitz} that
		\begin{align*}
                  \norm{\cG(x)-\cG_N(x)}_\cY^2
                  &=\norm{\sum_{j>N} \inp{\cG(x)}{\widetilde\eta_j}_\cY \eta_j}_\cY^2 \\
                  &\le \gL_\cY \sum_{j>N} \inp{\cG(x)}{\widetilde\eta_j}_\cY^2 w_j^{2t-2t} \\
                  &\le \gL_\cY w_{N+1}^{2t} \sum_{j>N} \inp{\cG(x)}{\widetilde\eta_j}_\cY^2 w_j^{-2t} \\
                  &\le \gL_\cY w_{N+1}^{2t} \norm{\cG(x)}_{\cY^t}^2,
		\end{align*}
		where the second inequality holds 
                since $\bw$ is a sequence of decreasing positive real numbers.
		Furthermore, by~\eqref{eq:lin_growth}, 
                since $\mathbf\Psi_\cY$ is a Riesz basis and due to $s>\frac{1}{2}$,
                it holds
		\begin{align*}
			\norm{\cG(x)}_{\cY^t}^2
			\le C^{2}(1+\norm{x}_\cX)^{2}
			\le C^2\left(1+\gL_\cX \sum_{i\in\bN}\inp{x}{\widetilde\psi_i}_\cX^2\right)
			\le C^2\left(1+\gL_\cX r^2\sum_{i\in\bN}w_i^{2s}\right)
		\end{align*}
        where the last term is finite and independent of $x\in C_r^s(\cX)$. This shows the claim.
	\end{proof}
	\subsection{Encoding Dimension Truncation}
	\label{sec:input_trunc}
	%
	Assume that we have fixed a truncation index $N$ for $\cG_N$
        in Subsection~\ref{sec:output_trunc}.
	We now choose \emph{component-dependent truncations of the encoded inputs}
        $M_j\in\bN$ for component $j$ of $G$ and 
        define corresponding scalar-valued mappings
	\begin{equation}\label{eq:g_j}
          g_j:\ell^2(\bN)\to \bR, \quad \bc\mapsto 
          \inp{[\cG\circ\cD_\cX\circ \cR_{M_j}](\bc)}{\widetilde\eta_j}_\cY
          =
          \left\langle\cG\left(\sum_{i=1}^{M_j}c_i\psi_i\right), \widetilde\eta_j\right\rangle_\cY \;, 
          \;\; j=1,\dots,N \;.
	\end{equation}
	We then define the multi-index (of truncation indices)
        $\mathbf{M}:=(M_1,\dots, M_N)\in\bN^N$ and the
        \emph{input-truncated approximation}
        $\cG_N^\mathbf{M}:\cX\to\cY$ to $\cG_N$ 
        via
	\begin{equation}\label{eq:output_trunc}
          \cG_N^\mathbf{M}:\cX\to\cY, \quad x\mapsto 
          \cD_{\cY}\left(([g_1\circ \cE_\cX](x), \dots, [g_N\circ \cE_\cX](x), 0, 0, \dots)\right)
          =
          \sum_{j=1}^N g_j(\cE_\cX(x))\eta_j. 
	\end{equation} 
        We observe that \eqref{eq:output_trunc} has the
        ``encoder-decoder'' structure that appears in a number of
        recently considered DON architectures.  We refer to the
        nonlocal neural operators \cite[Eqn.~(2.5) and
          App.~A.4]{lanthaler2023nonlocal},
        \cite{lanthaler2023operator}, to the deep ONets
        \cite[Eqn.~(33)]{NeurOp}, and to the references in
        \cite{NeurOp} for further DON architectures.  The next result
        has general expression rate bounds for
        \eqref{eq:output_trunc} which will apply in particular to
        the mentioned settings.

	\begin{prop}\label{prop:input_trunc}
          Under Assumption~\ref{ass:lipschitz}, for any $\gd>0$ exists
          a constant $C>0$, such that for any $\mathbf{M}\in\bN^N$
          there holds
          \begin{equation}\label{eq:input_trunc1}
            \sup_{x\in C_r^s(\cX)} \norm{\cG_N(x)-\cG_N^\mathbf{M}(x)}_\cY 
            \le C 
            \max_{j=1,\dots,N}  w_j^{t-\frac{1}{2}-\gd} w_{M_j+1}^{s-\frac{1}{2}-\gd}.
          \end{equation}
          In case that $\mathbf{M}=(M,\dots, M)\in\bN^N$ for some $M\in \bN$, 
          \begin{equation}\label{eq:input_trunc2}
            \sup_{x\in C_r^s(\cX)} \norm{\cG_N(x)-\cG_N^\mathbf{M}(x)}_\cY \le C w_{M}^{s-\frac{1}{2}-\gd}.
          \end{equation}
	\end{prop}
	
	\begin{proof}
          For any $x\in C_r^s(\cX)$ it holds
          by~\eqref{eq:output_trunc} and~\eqref{eq:g_j} that
          \begin{equation}
            \begin{split}\label{eq:input1}
              \norm{\cG_N(x)-\cG_N^\mathbf{M}(x)}_\cY^2 &=
              \norm{\sum_{j=1}^N
                \inp{\cG(x)}{\widetilde\eta_j}_\cY\eta_j
                - g_j(\cE_\cX(x))\eta_j}_\cY^2 \\
              &\le \gL_\cY \sum_{j=1}^N \abs{\inp{\cG(x)}{\widetilde\eta_j}_\cY-g_j(\cE_\cX(x))}^2 \\
              &=\gL_\cY \sum_{j=1}^N
              \left\langle\cG(x)-\cG\left(\sum_{i=1}^{M_j}\inp{x}{\widetilde\psi_i}_\cX\psi_i\right), \widetilde\eta_j\right\rangle_\cY^2
              w_j^{2t-2t}.
            \end{split}
          \end{equation}
          Now, for arbitrary $\mathbf M =(M_j)_{j=1}^N\in\bN^N$,
          the Lipschitz continuity of $\cG$ on $C_r^s(\cX)$ yields for
          any $\gd>0$ that
          \begin{align*}
            \norm{\cG_N(x)-\cG_N^\mathbf{M}(x)}_\cY^2
            &\le \gL_\cY \sum_{j=1}^N w_j^{2t} \norm{\cG(x)-\cG\left(\sum_{i=1}^{M_j}\inp{x}{\widetilde\psi_i}_\cX\psi_i\right)}_{\cY^t}^2\\
            &\le \gL_\cY L_\cG^2 \sum_{j=1}^N w_j^{2t}
              \norm{\sum_{i>M_j}\inp{x}{\widetilde\psi_i}_\cX\psi_i}_\cX^2\\
            &\le \gL_\cY L_\cG^2 \sum_{j=1}^N w_j^{2t}
              \gL_\cX \abs{\sum_{i>M_j}\inp{x}{\widetilde\psi_i}_\cX^{2}} \\
            &\le \gL_\cY L_\cG^2 \gL_\cX r^2\sum_{j=1}^N w_j^{2t}
              \sum_{i>M_j} w_i^{2s} \\
            &\le C \sum_{j=1}^N w_j^{2t} w_{M_j+1}^{2(s-\frac{1}{2}-\gd)} \sum_{i>M_j} w_i^{1+2\gd}\\
            &\le C \max_{j=1,\dots, N} w_j^{2(t-\frac{1}{2}-\gd)} w_{M_j+1}^{2(s-\frac{1}{2}-\gd)},
          \end{align*}
          which shows \eqref{eq:input_trunc1}.

            To show \eqref{eq:input_trunc2}, fix $M\in\bN$ and $\mathbf M=(M\dots, M)$.
            Then, using \eqref{eq:input1} and~\eqref{eq:norm_equiv}, 
          for any $\gd>0$
          \begin{align*}
            \norm{\cG_N(x)-\cG_N^\mathbf{M}(x)}_\cY^2
            &\le \gL_\cY \gl_\cY^{-1}\norm{\cG(x)-\cG\left(\sum_{i=1}^{M}\inp{x}{\widetilde\psi_i}_\cX\psi_i\right)}_{\cY^t}^2\\
            &\le \gL_\cY \gl_\cY^{-1} L_\cG^{2}\norm{\sum_{i>M_j}\inp{x}{\widetilde\psi_i}_\cX\psi_i}_\cX^2\\ 
            &\le \gL_\cY \gl_\cY^{-1} L_\cG^{2} \gL_\cX \abs{\sum_{i>M}\inp{x}{\widetilde\psi_i}_\cX^{2}} \\
            &\le C w_{M+1}^{2(s-\frac{1}{2}-\gd)}.
          \end{align*}
          This concludes the proof.
	\end{proof}

	\begin{rem}\label{rmk:TrncChoice}
        Proposition~\ref{prop:input_trunc} suggests 
          the following strategy to choose the truncation parameters
          $M_j$:
          \begin{itemize}
          \item $t>\frac 1 2$: For $\gd>0$ small enough it holds
            $t-\frac12-\gd>0$ and thus $w_j^{t-\frac12-\gd}\le 1$ for
            all $j\in\mathbb{N}$. Hence \eqref{eq:input_trunc1} is sharper
            than \eqref{eq:input_trunc2}, and a choice for
            $\mathbf{M}$ minimizing $\sum_{j=1}^N M_j$ while not
            increasing the best upper bound in
            \eqref{eq:input_trunc1}-\eqref{eq:input_trunc2} is
            obtained if
            $w_j^{t-\frac12-\gd}w_{M_j+1}^{s-\frac12-\gd}\sim {\rm
              const}$.
          \item $t\le \frac 1 2$: In this case $t-\frac12-\gd < 0$ and
            thus $w_j^{t-\frac12-\gd} \ge 1$ for all $j\in\mathbb{N}$.
            Hence \eqref{eq:input_trunc2} is sharper than
            \eqref{eq:input_trunc1}, and therefore a choice for
            $\mathbf{M}$ minimizing $\sum_{j=1}^N M_j$ while not
            increasing the best upper bound in
            \eqref{eq:input_trunc1}-\eqref{eq:input_trunc2} is
            obtained if $M_j=M$ for all $j$.
          \end{itemize}
\end{rem}
          \section{Deep Operator Surrogates}
          \label{sec:nn_approx}
	%
          The dimension truncation in $\cX$ and $\cY$ from the
          previous section effectively yields a finite-dimensional
          approximation $\cG_N^\mathbf{M}$ to $\cG$.  In the next
          step, we replace the dimension-truncated,
          finite-parametric, nonlinear coordinate map $\cG_N^\mathbf{M}$ 
          by an \emph{approximator}, i.e., 
          by neural network surrogate maps, and
          bound the resulting overall approximation error.  
          We further
          estimate the number of parameters (degrees of freedom) in
          the network, that are necessary to achieve a prescribed
          error tolerance $\eps>0$.
	
          We fix the following assumption on the approximation of
          $d$-variate Lipschitz functions.
          \begin{assumption}\label{ass:surrogate}
            Let $d\in\bN$ and $f:[0,1]^d\to\bR$ be a Lipschitz
            continuous function with Lipschitz constant $L_f>0$.

            Then, for any $\eps\in(0,1]$, 
            there exists a (neural network) surrogate 
            $\widetilde f:[0,1]^d\to\bR$ 
            with at most
            $\cO(d^{\ga}\eps^{-\gb}(1+\log(d)+\abs{\log(\eps)})^{\gk})$
            many parameters such that
            \begin{equation*}
              \norm{f-\widetilde f\,}_{L^2([0,1]^d)}\le L_f \eps.
            \end{equation*}
            The constants $\ga\ge 1$ and $\gb,\gk\ge 0$
            and the hidden constant in $\cO(\cdot)$ are
            independent of $d$, $L_f$ and $\eps$.
          \end{assumption}
	
          Assumption~\ref{ass:surrogate} holds for several DNN
          architectures and activations.  A (non-exhaustive)
          collection of examples is provided in
          Table~\ref{table:NN-architectures}.
          \begin{table}[ht]
            \centering
            \begin{tabular}{ |c||c|c|c|c||c|c|| c| }
              \hline
              Architecture & Activations & $\#$ of parameters & Assumption~\ref{ass:surrogate} \\
              \hline\hline
              Feedforward DNNs \cite{yarotsky2017error} & ReLU & 
         $\cO(d^2(1+\log(d))^2\eps^{-d}(1+\abs{\log(\eps)})^2)$ & \xmark \\ \hline
              NestNets \cite{shen2022neural} & ReLU &
      $\cO(\mathfrak h^2d^{2+d/(2(\mathfrak h+1))}\eps^{-d/(\mathfrak h+1)})$ for any $\mathfrak h\in\bN$ & (\cmark) \\ \hline 
              FLES \cite{shen2021neural} & $\floor{\cdot}$, $2^x$, $\sin$ &
                $\cO(d(1+\log(d) + \abs{\log(\eps)}))$ & \cmark \\ \hline
              Deep Fourier \cite{yarotsky2020phase} & ReLU, $\sin$ &
                 $\cO(d{(1+\abs{\log(\eps)})^2})$ & \cmark \\ \hline
            \end{tabular}
            \caption{\small Complexity of several neural network
              architectures to approximate Lipschitz functions
              $f:[0,1]^d\to \bR$ for given dimension $d\in\bN$ in
              $L^2([0,1]^d)$ with tolerance $\eps\in(0,1]$. 
              The parameter
              $\mathfrak h\in\bN$ in the second row signifies the
              ``height'' of a NestNet architecture, see also
              Example~\ref{ex:NestNet}.  }
            \label{table:NN-architectures}
          \end{table}
	
          \begin{ex}[NestNets] \label{ex:NestNet} In
            \cite{shen2022neural} the authors introduced a novel
            neural network architecture, dubbed \emph{NestNets}.  
            The
            proposed networks augment the classical, two-dimensional
            feedforward network (fully connected with size parameters
            width and depth) by a third dimension, termed ``height''
            in \cite{shen2022neural}, and indexed by
            ${\mathfrak h}\in\bN$.  Classical feedforward networks are
            contained as NestNets of height ${\mathfrak h}=1$.  By
            a bit-extraction technique, the authors showed
            that strict ReLU activated NestNets can express
            high-dimensional Hölder-continuous functions
            $f:[0,1]^d\to\bR$ in $L^p$ for $p<\infty$ without the CoD,
            \emph{in terms of the number of neurons constituting the
              NestNet}.  
            Moreover, by increasing the height
            ${\mathfrak h}$ of the network, one obtains in principle
            arbitrary fast algebraic rates of convergence with respect
            to the number of neurons which are free of the CoD,
            cf. Table~\ref{table:NN-architectures}.  
            In the NestNet architecture of \cite{shen2022neural}, 
            super-expressivity holds with strictly ReLU activated NestNets. 
            I.e., super-expressivity is afforded by 
            architecture, specifically the vastly larger
            connectivity of NestNets with height ${\mathfrak h}>1$, 
            rather than by more sophisticated activations.
            However, in one forward pass through a NestNet with height
            ${\mathfrak h}\ge 2$, the parameters are applied
            repeatedly to transform the input, with the number of
            floating points operations in a forward evaluation of the
            NestNet growing exponentially in ${\mathfrak h}$.  Thus,
            strictly speaking, \emph{NestNets overcome the CoD with
              respect to the number of neurons, but they do not lift
              the curse of dimensionality with respect to the number
              of floating point operations}.
          \end{ex}
	
\begin{rem}[NestNets and Skip Connections]
  NestNets of height ${\mathfrak h}>1$ could be viewed as extreme
  cases of feedforward NNs with skip connections.  
  Insertion of skip connections in, for example, feedforward DONs 
  has been \emph{empirically}, i.e., in numerical tests, 
  found to enhance DON expressivity significantly.  
  See, e.g.\ \cite[Section~2.4]{FineTunSNOs}.
\end{rem}

 \begin{ex}[Superexpressive activations] 
   In \cite{shen2021neural} the authors introduce so-called
   \emph{Floor-Exponential-Step} (FLES) networks with three hidden
   layers and a combination of floor ($x\mapsto \floor{x}$),
   exponential $x\mapsto 2^x$ and binary step units
   $x\mapsto\indi_{x\ge 0}$ as activation functions. Relying on a
   similar bit extraction technique as for NestNets, the authors show
   that for any Lipschitz function $f:[0,1]^d\to\bR$ there is a
   FLES-NN $\widetilde f:[0,1]^d\to\bR$ with at most $\cO({d+N})$
   parameters such that
   \begin{equation*}
     \|f-\widetilde f\|_{L^\infty([0,1]^d)}\le 6L_f\sqrt{d}2^{-N},\quad N\in\bN ,
   \end{equation*}
   see \cite[Corollary 1.2]{shen2021neural}.  
   
   For fixed $\eps\in(0,1]$ and $d\in\bN$, 
   	let $N_\eps:=\ceil{\log_2(6\sqrt{d})+|\log_2(\eps)|}\in\bN$, 
   	such that there holds $6\sqrt{d}2^{-N_\eps}\le \eps$. 
   	Then, there is a $C>0$ such that there holds
   	$$
        N_\eps\le 1+\log_2(6\sqrt{d}) + |\log_2(\eps)| \le C(1+\log(d)+|\log(\eps)|).
        $$
 	The total number of parameters $\widetilde N\in\bN$ of the FLES-NN $\widetilde f$ 
        is bounded by 
$$
        \widetilde N \le C(d+N) 
 	\le C(d+ 1+\log(d)+|\log(\eps)| )
 	\le Cd(1+\log(d)+\abs{\log(\eps)}) .
$$
   	Consequently,
   	Assumption~\ref{ass:surrogate} holds for FLES feedforward NNs with
   	$\ga =\gk = 1$ and $\gb=0$.

   This "superexpressivity" has been generalized for various
   elementary activation functions in~\cite{yarotsky2021elementary}.
   Therein, the author shows that the approximation technique for FLES
   in \cite{shen2021neural} may be transferred to several classes of
   elementary, smooth activations.  The results
   in~\cite{yarotsky2021elementary} are stated for general, uniformly
   continuous functions $f:[0,1]^d\to\bR$ with respect to the supremum
   norm, thus no rates are derived.  However, since the line of proof
   closely follows~\cite{shen2021neural}, one would expect the same
   exponential rates for feed forward networks with super expressive
   smooth activations.  On a further note, Yarotski emphasizes
   in~\cite[Section 3]{yarotsky2021elementary} that most standard
   activation functions like ReLU, $\tanh$, sigmoid, binary step
   units, etc.\ are \emph{not} superexpressive.
 \end{ex}

 To construct the surrogate operator networks, we approximate for each
 $j$ the parametric co-ordinate maps $g_j:\ell^2(\bN)\to\bR$
 in~\eqref{eq:g_j}.  In order to "restrict" $g_j$ to the
 finite-dimensional domain $\bR^{M_j}$, we introduce the restriction
 maps
 \begin{equation}\label{eq:coordinate_g_j}
   \mfg_j:\bR^{M_j}\to \bR, \quad 
   (c_1,\dots, c_n)\mapsto 
   \left\langle\cG\left(\sum_{i=1}^{M_j}c_i\psi_i\right), \widetilde\eta_j\right\rangle_\cY, 
   \qquad j=1,\dots, N.
 \end{equation}
 We further introduce the \emph{projection operator} $\Pi_{M_j}$ via
 \begin{equation*}
   \Pi_{M_j}:\ell^2(\bN)\to \bR^{M_j}, \quad \bc \mapsto (c_1,\dots, c_{M_j}),
 \end{equation*}
 to obtain the identity
 \begin{equation}\label{eq:g_j_relation}
   g_j(\bc) = \mfg_j(\Pi_{M_j}(\bc)), \quad c\in \ell^2(\bN).
 \end{equation}
 In the next step, we establish the Lipschitz continuity of $\mfg_j$.
 \begin{lem}\label{lem:g_Lipschitz}
   Let Assumption~\ref{ass:lipschitz} hold and let
   $\mfg_j:\bR^{M_j}\to\bR$ be defined as in~\eqref{eq:coordinate_g_j}
   for a given truncation index $M_j\in\bN$.  
   Then, for any $c$,
   $c' \in \bigtimes_{i=1}^{M_j}[-rw_i^s, rw_i^s]\subset \bR^{M_j}$
   there holds
   \begin{equation}\label{eq:lipschitz2}
     \abs{\mfg_j(c)-\mfg_j(c')}
     \le L_\cG w_j^t 
     {\sqrt{\gL_\cX}}
     \|c-c'\|_2.
   \end{equation}
 \end{lem}
 \begin{proof}
   Fix $c$, $c' \in \bigtimes_{i=1}^{M_j}[-rw_i^s, rw_i^s]$.  Then
   with $x:=\sum_{i=1}^{M_j}c_i\psi_i$,
   $x':=\sum_{i=1}^{M_j}c_i'\psi_i\in C_r^s(\cX)$ holds
   $c=\Pi_{M_j}(\cE_\cX(x))$ and $c'=\Pi_{M_j}(\cE_\cX(x'))$. Thus
   \begin{equation*}
     \abs{\mfg_j(c)-\mfg_j(c')}
     =
     \abs{\mfg_j([\Pi_{M_j}\circ \cE_\cX](x))-\mfg_j([\Pi_{M_j}\circ \cE_\cX](x'))}
     =
     \abs{g_j(\cE_\cX(x))-g_j(\cE_\cX(x'))}
   \end{equation*}
   holds by~\eqref{eq:g_j_relation}.
   Furthermore,
   \begin{align*}
     \abs{g_j(\cE_\cX(x))-g_j(\cE_\cX(x'))}^2 
     &=
       \left\langle \cG\left(\sum_{i=1}^{M_j}\inp{x}{\widetilde\psi_i}_\cX\psi_i\right)
       	-
       	\cG\left(\sum_{i=1}^{M_j}\inp{x'}{\widetilde\psi_i}_\cX\psi_i\right),
       	\widetilde\eta_j\right\rangle_\cY^2 
      \\
     &\le w_j^{2t} \norm{\cG\left(\sum_{i=1}^{M_j}\inp{x}{\widetilde\psi_i}_\cX\psi_i\right)-\cG\left(\sum_{i=1}^{M_j}\inp{x'}{\widetilde\psi_i}_\cX\psi_i\right)}_{\cY_t}^2 \\
     &\le w_j^{2t} L_\cG^2 \norm{\sum_{i=1}^{M_j}\inp{x-x'}{\widetilde\psi_i}_\cX\psi_i}_{\cX} 
     \\
     &\le w_j^{2t} L_\cG^2 \gL_\cX \sum_{i=1}^{M_j}\inp{x-x'}{\widetilde\psi_i}_\cX^2 
     \\
     &= w_j^{2t} L_\cG^2 \gL_\cX \sum_{i=1}^{M_j}\abs{\Pi_{M_j}(\cE_\cX(x))_i-\Pi_{M_j}(\cE_\cX(x'))_i}^2 
     \\
     &= w_j^{2t} L_\cG^2 \gL_\cX\norm{c-c'}_2^2.
   \end{align*}
 \end{proof}
	
	\begin{lem}\label{lem:lipschitz_approx}
          Let Assumption~\ref{ass:surrogate} hold, let $M\in\bN$ be
          arbitrary and define
          $ D_M:=\bigtimes_{i=1}^{M}[-rw_i^s, rw_i^s]$.  Denote by
          $\gl$ the univariate Lebesgue measure on $(\bR, \cB(\bR))$
          and by $\mu_M:=\bigotimes_{i=1}^M \frac{\gl}{2rw_i^s}$ the
          uniform probability measure on $ D_M \subset \bR^M$.
          Further, assume $g:D_M\to \bR$ is Lipschitz continuous with
          Lipschitz constant $L_g>0$.
          Then, for any $\eps\in(0,1]$, 
          there exists a neural network $\widetilde g:D_M\to \bR$ 
          with at most
          $\cO{(M^\ga\eps^{-\gb}(1+\log(M)+|\log(\eps)|)^\gk)}$ parameters, 
          such that
          \begin{equation*}
            \norm{g-\widetilde g}_{L^2( D_M, \mu_M)}\le L_g 2r w_1^s \eps.
          \end{equation*}
	\end{lem}
	\begin{proof}
          We first translate the unit cube $[0,1]^M$ to $ D_M$
          with the linear bijection
          \begin{equation*}
            T:[0,1]^M\to D_M , \quad 
            (u_1, \cdots, u_M)\mapsto r\left((2u_1-1)w_1^s, \cdots, (2
              u_M
              -1)
              w_M^s
            \right). 
          \end{equation*}
          Set $g_T:=g\circ T:[0,1]^M\to \bR$. Clearly, 
          for all $u$, $u'\in[0,1]^M$ 
          \begin{align*}
            \abs{g_T(u)-g_T(u')}\le L_g \|T(u)-
            T(u')
            \|_2
            = L_g 2 r \left(\sum_{i=1}^M\abs{u_i-u'_i}^2w_i^{2s}\right)^{1/2}
            \le L_g 2 r w_1^s \|u-u'\|_2.
          \end{align*}
          By Assumption~\ref{ass:surrogate}, for any
          $\eps\in(0,1]$ and for any finite $M\in \bN$ 
          exists a neural network $\widetilde g_T:[0,1]^M\to \bR$ 
          with at most $C (M^\ga\eps^{-\gb}(1+\log(M)+|\log(\eps)|)^\gk)$ 
          parameters, where the constant $C>0$ 
          is independent of $\eps$ and $M$, 
          and
          \begin{equation*}
            \|g_T-\widetilde g_T\|_{L^2([0,1]^M)}\le L_g 2r w_1^s \eps.
          \end{equation*} 
          Denote by $\gl^M$ 
          the Lebesgue measure on $([0,1]^M, \cB([0,1]^M))$.  
          For the pushforward measure 
          $T_\#\gl^M:= \gl^M\circ T^{-1}$ on $( D_M, \cB( D_M))$ and 
          $\widetilde g:=\widetilde g_T\circ T^{-1}$ there holds
          \begin{equation*}
            \|g_T-\widetilde g_T\|_{L^2([0,1]^M)}
            =
            \|g\circ T-\widetilde g\circ T\|_{L^2([0,1]^M)}
            =
            \|g-\widetilde g\|_{L^2( D_M, T_\#\gl^M)}.
          \end{equation*}
          Since $T_\#\gl^M=\bigotimes_{i=1}^M
          \frac{\gl}{2rw_i^s}=\mu_M$ it thus follows that
          \begin{equation*}
            \|g-\widetilde g\|_{L^2( D_M, \mu_M)}
            \le L_g 2r w_1^s \eps.
          \end{equation*}
          The linear transformation $T^{-1}$ introduces
          $M$ additional parameters, regardless of
          $\eps$. 
          But since $\eps\le 1$ and $\ga\ge1$, 
          it holds that $\widetilde g$ 
          has at most $\cO(M^\ga\eps^{-\gb}(1+\log(M)+|\log(\eps)|)^\gk)$  
          parameters.
	\end{proof}
	
	To bound the mean-squared error of the overall approximation,
        we introduce the uniform product probability measure
        $\cP_U:=\bigotimes_{i\in\bN}
        \frac{\gl}{2}$ on $U:=[-1,1]^\bN$, where
        $\gl$ is the Lebesgue measure on $(\bR, \cB(\bR))$ and
        $U$ is equipped with the product Borel
        $\gs$-algebra $\cB([-1,1]^\bN)$.
	We further define the random variable $\gs_r^s:U\to \cX$ 
        on $(U, \cB([-1,1]^\bN), \cP_U)$ 
        via
	\begin{equation*}
          \gs_r^s:U\to \cX, \quad \mfu \mapsto r\sum_{i\in\bN} w_i^s \mfu_i\psi_i,
	\end{equation*}
	and note that $C_r^s(\cX)\subseteq
        \gs_r^s[U]$.  On the other hand,
        $\cE_\cX(\gs_r^s({\mfu}))=(rw_i^s{\mfu}_i, i\in\bN)\in
        \bigtimes_{i\in\bN}[-rw_i^s, rw_i^s]$ since
        $\mathbf\Psi_\cX$ is a Riesz basis, and thus $C_r^s(\cX)=
        \gs_r^s[U]$.  Hence, the pushforward measure
        $(\gs_r^s)_\#\cP_U$ is supported on $C_r^s(\cX)$.
	Our main result gives a bound on $\eps$-complexity of expression 
    for Lipschitz maps, in terms of the number $ \cN_{para}$ of neurons
    which are sufficient for $\eps$-consistency of the DON.
	\begin{thm}\label{thm:MainRes}
		Let Assumptions~\ref{ass:lipschitz}
		and~\ref{ass:surrogate} hold and fix $\delta>0$
		(arbitrarily small).
		Then, for any $\eps \in (0,1]$ exists a neural
		network $\widetilde G:\ell^2(\bN)\to\ell^2(\bN)$ with at
		most $\cN_{para}\in\bN$ parameters such that
		\begin{equation*}
			\|\cG-\cD_\cY\circ \widetilde G\circ \cE_\cX\|_{L^2(C_r^s(\cX), (\gs_r^s)_\#\cP_U; \cY)} 
			\le \eps,
		\end{equation*}
		and , for some constant $C>0$ independent of $\eps$
		\begin{equation}\label{eq:dof_estimate}
			\cN_{para}\le 
			C
			\begin{cases}
				\eps^{-\ga/(s-1/2)-(2+\gb)/2t-\delta}(1+\abs{\log(\eps)})^\gk ,
				&\quad t\le \frac{1}{2}, \\
				{\eps^{-\ga/(s-1/2)-1/t-\gb-\delta}(1+\abs{\log(\eps)})^\gk},
				&\quad t> \frac{1}{2},
			\end{cases}
		\end{equation}
		where $C=C(\delta, L_\cG)>0$ is independent of $\eps$.
	\end{thm}
	
	\begin{proof}
		We first prove the claim in case that $t\le \frac{1}{2}$.
		
		Fix $\eps\in(0,1]$ and 
		recall from ~\eqref{eq:output_trunc}
		that
		$\cG_N^{\mathbf M}=\sum_{j=1}^N g_j(\cE_\cX(\cdot))\eta_j$
		for any $N$ and $\mathbf{M}\in\bN^N$.  We construct
		$\widetilde\cG$ by first choosing appropriate truncation
		indices $N$ and $\mathbf{M}\in\bN^N$ to obtain
		$\cG_N^{\mathbf M}$ such that
		$\sup_{x\in C_r^s(\cX)}\|\cG(x)-\cG_N^{\mathbf M}(x)\|_\cY
		\le\frac{\eps}{2}$.  Then we substitute each map $g_j$ in
		$\cG_N^{\mathbf M}$ by an appropriate neural network
		surrogate to achieve an overall error of at most $\eps$.
		
		1.) \emph{Input and output truncation:}
		Propositions~\ref{prop:input_trunc}
		and~\ref{prop:output_trunc} show that for any
		$\gd_0\in(0,s-\frac{1}{2})$ there exists
		$C_{L_\cG,\gd_0}>0$ such that for
		$\mathbf M=(M, \dots, M)$ where 
		$M\in\bN$ and $N\in\bN$ there holds
		\begin{equation}\label{eq:comb_trunc_err2}
			\sup_{x\in C_r^s(\cX)}\|\cG(x)-\cG_N^{\mathbf M}(x)\|_\cY 
			\le C_{L_\cG,\gd_0}\left(w_{N+1}^{t} + w_{M+1}^{s-\frac{1}{2}-\gd_0}\right).
		\end{equation}
		As the weight sequence $\bw\in (0,1]^\N$ is non-increasing
		with $\bw\in\ell^{1+\widehat\eps}(\N)$ for any
		$\widehat\eps>0$, there exists for any $\eps_0>0$ a $C_0\ge1$ such that
		$w_i\le C_0 i^{-1+\eps_0}$ for all $i\in\bN$, and thus,
		$C_{L_\cG,\gd_0} w_{N+1}^t\le C_{L_\cG,\gd_0} C_0^tN^{-t(1-\eps_0)}$.  
		Hence, we may fix some (arbitrary small) 
		$\gd_{0}\in(0,\frac{s}{2}-\frac{1}{4})$ and $\eps_0\in(0,\frac{\gd_0}{s-1/2-\gd_0})$,
		and let
		\begin{equation}\label{eq:N}
			N:=\ceil{\left(\frac{\eps}{4C_{L_\cG,\gd_0}C_0^t}\right)^{-\frac{1}{t(1-\eps_0)}}}\in\bN.
		\end{equation}
                Note that by the choice of $\gd_0$, $\eps_0$ there
                  holds $\eps_0<1$ and
                  $(s-\frac{1}{2}-\gd_0)(1-\eps_0)>s-\frac{1}{2}-2\gd_0>0$.
                Since $t\le \frac{1}{2}$, we thus set
                $\mathbf M=(M,\dots, M)\in\bN^N$ with
		\begin{equation}\label{eq:M}
			M:=\ceil{\left(\frac{\eps}{4C_{L_\cG,\gd_{0}} 
					C_0^{s-\frac{1}{2}-\gd_0}
				}\right)^{-\frac{1}{s-1/2-2\gd_{0}}}}
			\ge 
			\ceil{\left(\frac{\eps}{4C_{L_\cG,\gd_{0}} C_0^{s-\frac{1}{2}-\gd_0} }\right)^{-\frac{1}{(s-1/2-\gd_{0})(1-\eps_0)}}}.
		\end{equation}
		Combing the choices of $N$ and $\mathbf M=(M,\dots, M)$
		with~\eqref{eq:comb_trunc_err2} and $w_i\le C_0 i^{-1+\eps_0}$ then shows
		\begin{equation}\label{eq:eps-half-estimate}
			\sup_{x\in C_r^s(\cX)}\|\cG(x)-\cG_N^{\mathbf M}(x)\|_\cY 
			\le \frac{\eps}{4}+\frac{\eps}{4}=\frac{\eps}{2}.
		\end{equation}
		
		2.) \emph{Neural network surrogates for $g_j$:} It holds
		by~\eqref{eq:g_j_relation} that
		$\cG_N^{\mathbf M}=\sum_{j=1}^N
		g_j(\cE_\cX(\cdot))\eta_j=\sum_{j=1}^N
		\mfg_j(\Pi_{M_j}\circ\cE_\cX(\cdot))\eta_j$.
		Furthermore, Lemma~\ref{lem:g_Lipschitz} shows that each
		$\mfg_j:\bR^{M_j}\to \bR$ is Lipschitz continuous on the
		subset
		$D_{M_j}:=\bigtimes_{i=1}^{M_j}[-rw_i^s, rw_i^s]\subset
		\bR^{M_j}$ with Lipschitz constant given by
		$L_\cG w_j^t \sqrt{\gL_\cY}$.
		Lemma~\ref{lem:lipschitz_approx} then in turn shows that for
		any fixed $j=1,\dots, N$ and for any $\eps_j\in(0,1]$ 
		there
		exists an approximation $\widetilde{\mfg}_j:D_{M_j}\to \bR$
		of $\mfg_j$, such that
		\begin{equation}\label{eq:NN_estimate}
			\|\mfg_j-\widetilde{\mfg}_j\|_{L^2(\cD_{M_j}, \mu_{M_j})}
			\le 
			L_\cG w_j^t \sqrt{\gL_\cY} 2r w_1^s \eps_j,
		\end{equation}
		where
		$\mu_{M_j}:=\bigotimes_{i=1}^{M_j} \frac{\gl}{2rw_i^s}$
		denotes the uniform probability measure on $D_{M_j}$.
		Furthermore, the DNN $\widetilde{\mfg}_j$ uses at most
		$\cO(M_j^\ga\eps_j^{-\gb}(1+\log(M_j)+|\log(\eps_j)|)^{\gk})$
		parameters.
		
		With this at hand we define the neural network surrogate
		\begin{equation}
			\widetilde\cG:\cX\to \cY, \quad x\mapsto \sum_{j=1}^N \widetilde{\mfg}_j(\Pi_{M_j}\cE_\cX(x))\eta_j.
		\end{equation}
		To bound the error in this surrogate, we observe that
		for any $x\in C_r^s(\cX)$ it holds that
		\begin{align*}
			\norm{\cG_N^{\mathbf M}(x)-\widetilde\cG(x)}_\cY^2
			&= \norm{\sum_{j=1}^N \left(\mfg_j(\Pi_{M_j}\cE_\cX(x)) - \widetilde{\mfg}_j(\Pi_{M_j}\cE_\cX(x))\right)\eta_j}_\cY^2 \\
			&\le \gL_\cY \sum_{j=1}^N 
			\abs{\mfg_j(\Pi_{M_j}\cE_\cX(x)) - \widetilde{\mfg}_j(\Pi_{M_j}\cE_\cX(x))}^2.
		\end{align*}
		Furthermore, we have
		$\Pi_{M_j}\cE_\cX(x)= \Pi_{M_j}\cE_\cX(\gs_r^s(\mfu))$
		for some (in general non-unique) $\mfu \in U$. 
		In addition, if ${\mfu}\sim\cP_U$ 
		then
		$\Pi_{M_j}\cE_\cX(\gs_r^s({\mfu}))=r(w_1^s \mfu_1, \dots,
		w_{M_j}^s \mfu_{M_j})\sim \bigotimes_{i=1}^{M_j}
		\frac{\gl}{2rw_i^s}=\mu_{M_j}$.
		Therefore,~\eqref{eq:NN_estimate} yields with
		$C_{L_\cG}:=L_\cG \gL_\cY^{\frac{3}{2}}2rw_1^s$ 
		that
		\begin{align*}
			\norm{\cG_N^{\mathbf M}-\widetilde\cG}_{L^2(C_r^s(\cX), (\gs_r^s)_\#\cP_U; \cY)}^2
			\le \gL_\cY \sum_{j=1}^N 
			\norm{\mfg_j - \widetilde{\mfg}_j}_{L^2(D_{M_j}, \mu_{M_j})}^2
			\le C_{L_\cG}^2\sum_{j=1}^N w_j^{2t}\eps_j^2.
		\end{align*}
		We then obtain
		\begin{align*}
			\norm{\cG_N^{\mathbf M}-\widetilde\cG}_{L^2(C_r^s(\cX), (\gs_r^s)_\#\cP_U; \cY)}^2
			&\le C_{L_\cG}^2\left(\max_{j=1,\dots,N}  
			w_j^{2(t-\frac{1}{2}-\gd_{ 0})}\eps_j^2\right) \sum_{j=1}^N w_j^{1+2\gd_{0}} 
			\\
			&\le C_{L_\cG}^2C_{0}^{1+2\gd_{0}}\zeta(1+2\gd_{0})
			\left(\max_{j=1,\dots,N}  w_j^{2(t-\frac{1}{2}-\gd_{0})}\eps_j^2\right),
		\end{align*}
		where $\zeta:(1,\infty)\to(0,\infty)$ denotes the Riemann zeta function.
		Now let
		$C_{\gd_0} := C_{0}^{\frac{1}{2}+\gd_{0}}\zeta(1+2\gd_{0})^\frac{1}{2}$ and
		\begin{equation}\label{eq:epsilon_j}
			\eps_j:=\min\left(
			\frac{w_j^{-(t-\frac{1}{2}-\gd_{0})}\eps}
			{2C_{L_\cG} C_{\gd_0}}
			,\, 1\right)\in(0,1]
			, \quad j=1,\dots, N.
		\end{equation}
		The triangle inequality and~\eqref{eq:eps-half-estimate}
		then show
		\begin{align*}
			\norm{\cG-\widetilde\cG}_{L^2(C_r^s(\cX), (\gs_r^s)_\#\cP_U; \cY)}
			&\le
			\sup_{x\in C_r^s(\cX)}\|\cG(x)-\cG_N^{\mathbf M}(x)\|_\cY 
			+
			\norm{\cG_N^{\mathbf M}-\widetilde\cG}_{L^2(C_r^s(\cX), (\gs_r^s)_\#\cP_U; \cY)} \\
			&\le 
			\frac{\eps}{2}+\frac{\eps}{2}=\eps.
		\end{align*}
		
		3.) \emph{Overall complexity:} Let $\cN_{para}\in\bN$ denote
		the total number of parameters used to construct
		$\widetilde\cG$. We have in total used $N$ surrogates
		$\widetilde{\mfg}_j:D_{M_j}\to \bR$ with prescribed
		accuracies $\eps_j\in (0,1]$ for $j=1,\dots, N$.  
		Each of these
		surrogates therefore involves (at most)
		$\cO(M_j^\ga\eps_j^{-\gb}(1+\log(M_j)+|\log(\eps_j)|)^{\gk})$ 
		parameters 
		by Lemma~\ref{lem:lipschitz_approx}.
		Hence, 
		\begin{equation} \label{eq:Cdofs} 
			\cN_{para}\le C_{dofs}\sum_{j=1}^N M_j^\ga\eps_j^{-\gb}(1+\log(M_j)+|\log(\eps_j)|)^{\gk},
		\end{equation}
		where the constant $C_{dofs}>0$ is independent of $\eps_j, N$ and $M_j$.
		
		For fixed $\gd_0\in(0,s-\frac{1}{2})$ 
		independent of $\eps$, 
		we recall that $M_j=M$ for all $j$, where $M$ is given in~\eqref{eq:M}.  
		Substituting the
		choices of $M$ and $\eps_j$ in~\eqref{eq:M}
		and~\eqref{eq:epsilon_j}, respectively, yields
		\begin{align*}
			\cN_{para}
			\le C_{dofs} \sum_{j=1}^{N}& 
			\ceil{\left(\frac{\eps}{4C_{L_\cG,\gd_0}
					C_0^{s-\frac{1}{2}-\gd_0}
				}\right)^{-\frac{1}{s-1/2-2\gd_0}} \frac{1}{C_0}}^\ga
			\min \left(\frac{w_j^{-(t-\frac{1}{2}-\gd_0)}\eps}{2C_{L_\cG} C_{\gd_0}}, \, 1\right)^{-\gb}\\
			&\;\cdot\left(1 + 
			\log\left(\ceil{\left(\frac{\eps}{4C_{L_\cG,\gd_0}
					C_0^{s-\frac{1}{2}-\gd_0}
				}\right)^{-\frac{1}{s-1/2-2\gd_0}} 	\frac{1}{C_0}}\right)
			+
			\abs{\log\left(\frac{w_j^{-(t-\frac{1}{2}-\gd_0)}\eps}{2C_{L_\cG}  C_{\gd_0}}\right)}
			\right)^\gk , 
		\end{align*}
		where only the constants $C_{\gd_0}$ and
		$C_{L_\cG,\gd_0}$ depend on $\gd_0$.
		
		Further, recall that the constants $C_{L_\cG}$ and that 
		$C_{L_\cG,\gd_0}$ grow linearly with respect to $L_\cG$.
		Hence, $t-\frac{1}{2}-\gd_0<0$ and $N$ as
		in~\eqref{eq:N} show that for any $\eps_0\in(0,1)$ there holds
		\begin{align*}
			\cN_{para}
			&\le C \sum_{j=1}^{N} \eps^{-\frac{\ga}{s-1/2-2\gd_0}-\gb}w_j^{\gb(t-\frac{1}{2}-\gd_0)} 
			(1+\abs{\log(\eps)}+\abs{\log(\eps w_j)}|)^\gk	
			\\
			&\le C 
			N \eps^{-\frac{\ga}{s-1/2-2\gd_0}-\gb}
			\max_{j=1\,\dots, N}
			\left(w_j^{\gb(t-\frac{1}{2}-\gd_0)}\right) 
			(1+\abs{\log(\eps)}+\abs{\log(w_j)}|)^\gk \\
			&\le C  \eps^{-\frac{\ga}{s-1/2-2\gd_0}-\gb}         
			N^{1-\gb(t-\frac{1}{2}-\gd_0)(1-\eps_0)}
			(1+\abs{\log(\eps)}+\abs{\log(N)}|)^\gk\\
			&\le C  \eps^{-\frac{\ga}{s-1/2-2\gd_0}-\gb}        \eps^{(-1+\gb(t-\frac{1}{2}-\gd_0)(1-\eps_0))\frac{1}{t(1-\eps_0)} }
			(1+\abs{\log(\eps)})^\gk\\
			&\le C \eps^{-\frac{\ga}{s-1/2}-\frac{(2+\gb)}{2t}-\gd}
			(1+\abs{\log(\eps)})^\gk,
		\end{align*}
		where $C=C(\gd_0, \eps_0, L_\cG)$ is independent of $\eps$, and $\gd$ is given by
		$$\gd=\frac{\ga}{s-1/2-2\gd_0} +\gb + \frac{1}{t(1-\eps_0)} - \frac{\gb(2t-1-2\gd_0)}{2t}
		-\frac{\ga}{s-1/2}-\frac{(2+\gb)}{2t}
		=O(\gd_0)\qquad\text{as }\gd_0\to 0.$$       
		The claim for
		$t\le\frac{1}{2}$ now follows for arbitrary $\gd>0$ by
		choosing $\gd_0$ and $\eps_0$ sufficiently small.
		
		4.) \emph{Case $t>\frac{1}{2}$:} We only highlight the
		changes that are necessary in the proof for
		$t\leq \frac{1}{2}$.
		
		4a) Observe that for arbitrary  
		$\gd_0\in(0,\frac{\min(s,t)}{2}-\frac{1}{4})$ and $\eps_0\in(0,\frac{\gd_0}{\max(s,t)-1/2-\gd_0})$ there holds
		\begin{equation*}
			\max_{j=1,\dots,N}  w_j^{t-\frac{1}{2}-\gd_0} w_{M_j+1}^{s-\frac{1}{2}-\gd_0}
			\le
			C_0^{t+s-1-2\gd_0}
				\max_{j=1,\dots,N} 
				j^{-t+\frac{1}{2}+2\gd_0} M_j^{-s+\frac{1}{2}+2\gd_0}.
		\end{equation*}
		Hence, we set 
		\begin{equation}\label{eq:M_j}
			M_j:=
			\ceil{\left(\frac{\eps}{4 C_{L_\cG,\gd_0}}j^{t-\frac{1}{2}-2\gd_0}
				C_0^{t+s-1-2\gd_0}
				\right)^{-\frac{1}{s-1/2-2\gd_0}} }, \quad j=1,\dots, N.
		\end{equation}
		We further observe that $M_1\simeq M$ with $M$ as in~\eqref{eq:M}, 
		and that $M_j$ is decreasing with increasing $j$, 
		since $\frac{t-1/2-2\gd_0}{s-1/2-2\gd_0}>0$ as $\gd_0<\frac{\min(s,t)}{2}-\frac{1}{4}$.
		Next, by
		Propositions~\ref{prop:input_trunc}
		and~\ref{prop:output_trunc} for any
		$\mathbf{M}\in\bN^N$ there holds
		\begin{equation}\label{eq:comb_trunc_err1}
			\sup_{x\in C_r^s(\cX)}\|\cG(x)-\cG_N^{\mathbf M}(x)\|_\cY 
			\le 
			C_{L_\cG,\gd_0}
			\left(w_{N+1}^{t} + \max_{j=1,\dots,N}  
			w_j^{t-\frac{1}{2}-\gd_0} w_{M_j+1}^{s-\frac{1}{2}-\gd_0}\right).
		\end{equation}
		With $\mathbf M=(M_1, \dots, M_N)$ in~\eqref{eq:M_j}
		it then follows 
		that 
		\begin{equation*}
			\sup_{x\in C_r^s(\cX)}\|\cG(x)-\cG_N^{\mathbf M}(x)\|_\cY 
			\le \frac{\eps}{4}+\frac{\eps}{4}=\frac{\eps}{2}.
		\end{equation*}
		
		4b) 
		Let $\gd_0\in(0,\frac{\min(s,t)}{2}-\frac{1}{4})$ and $\eps_0\in(0,\frac{\gd_0}{\max(s,t)-1/2-\gd_0})$
		and use $M_j$ as in~\eqref{eq:M_j} for fixed $\gd_0$. 
		Thus, there holds $\log(M_j)\le C(1+\abs{\log(\eps)}+\log(j))$ 
		and we obtain the estimate
		\begin{align*}
			\cN_{para}
			&\le C \eps^{-\frac{\ga}{s-1/2-2\gd_0}-\gb} 
			\sum_{j=1}^{N} j^{-\ga\frac{t-1/2-2\gd_0}{s-1/2-2\gd_0}} w_j^{\gb(t-\frac{1}{2}-\gd_0)} 
			(1+\abs{\log(\eps)} + \log(j))^\gk\\
			&\le C \eps^{-\frac{\ga}{s-1/2-2\gd_0}-\gb} 
			\sum_{j=1}^{N} j^{-(\frac{\ga}{s-1/2-2\gd_0}+\gb(1-\eps_0))(t-\frac{1}{2}-2\gd_0)} 
			(1+\abs{\log(\eps)} + \log(j))^\gk\ \\
			&\le C \eps^{-\frac{\ga}{s-1/2-2\gd_0}-\gb}  N
			\left(\max_{j=1,\dots, N} j^{-(\frac{\ga}{s-1/2-2\gd_0}+\gb(1-\eps_0))(t-\frac{1}{2}-2\gd_0)}  
			\right)
			(1+\abs{\log(\eps)} + \log(N))^\gk.
		\end{align*}
		As $2\gd_0<\min(s,t)-\frac{1}{2}$ and $\eps_0\in(0,1)$, 
		it holds that $(\frac{\ga}{s-1/2-2\gd_0} + 
		\gb(1-\eps_0))(t-\frac{1}{2}-2\gd_0)\ge 0$ 
		and thus with $N$ as in~\eqref{eq:N} it follows that  
		\begin{align*}
			\cN_{para}
			&\le C \eps^{-\frac{\ga}{s-1/2-2\gd_0}-\gb}  N 
			(1+\abs{\log(\eps)} + \log(N))^\gk \\
			&\le C \eps^{-\frac{\ga}{s-1/2-2\gd_0}-\gb-\frac{1}{t(1-\eps_0)}}
			(1+\abs{\log(\eps)})^\gk.
		\end{align*}
		The claim then follows for arbitrary small 
		$$\gd=\frac{\ga}{s-1/2-2\gd_0}  +
		\frac{1}{t(1-\eps_0)}-\frac{\ga}{s-1/2} - \frac{1}{t}>0$$
		with appropriate choices of $\gd_0$ and $\eps_0$, 
		since $\gd=\cO(\gd_0+\eps_0)=\cO(\gd_0)$ as $\gd_0\to 0$.
	\end{proof}
\section{Discussion and Extensions}
\label{sec:Extns}
%
We indicate several
directions in which our main result, Theorem~\ref{thm:MainRes}, can be
extended. First, due to the invariance of the (metric) space of
Lipschitz maps under composition, Theorem~\ref{thm:MainRes} can be
applied to finite compositions of Lipschitz maps in a
``component-wise'' manner, with different truncation dimensions, and
different encoders and decoders at each stage of composition.  This
allows, in particular, to leverage known expression rate bounds for
the factor maps.  We indicate the ideas and details in
Section~\ref{sec:Compos}.
\subsection{Compositions of Lipschitz Maps}
\label{sec:Compos}
The preceding operator emulation bound readily applies 
to \emph{composition of operators}: 
let $K\geq 2$ be an integer,
and 
assume given a collection $(\cX_k, k=0,\dots, K)$ of Banach spaces,
with norms $\norm{\cdot}_{\cX_k}$.
Further, denote for any $k=1,...,K$ by $\Lip(\cX_{k-1},\cX_k)$ 
the set of all Lipschitz continuous mappings from $\cX_{k-1}$ to $\cX_k$.
For nonlinear maps $\cG_k \in \Lip(\cX_{k-1},\cX_k)\;,\;\; k=1,...,K$ 
consider the composition
\begin{equation}\label{eq:cGK}
\cG_{[1:k]} = \cG_k \circ ... \circ \cG_1, \quad k=1,...,K\;.
\end{equation}
Setting $\cX := \cX_0$ and $\cY:= \cX_K$, 
evidently $\cG_{[1:K]} \in \Lip(\cX,\cY)$.
The following observation is elementary.
\begin{lem}\label{lem:LipComp}
Assume 
each of the ``factor'' operators $\cG_k$ constituting $\cG_{[1:k]}$ 
in \eqref{eq:cGK}
are Lipschitz with Lipschitz constant
$$
L_k
:= 
\sup_{x,\widetilde{x}\in \cX_{k-1}} 
\frac{\| \cG_k(x) - \cG_k(\widetilde{x}) \|_{\cX_k}}{\| x-\widetilde{x} \|_{\cX_{k-1}}}\;,
\;\; 
k=1,...,K\;.
$$

Then, the composition $\cG_{[1:k]}$ in \eqref{eq:cGK} is in $\Lip(\cX_0,\cX_k)$
with Lipschitz constant $L_{[1:k]}$ bounded by
\begin{equation}\label{eq:LipProd}
	L_{[1:k]} \leq \prod_{j=1}^k L_j \;.
\end{equation}
\end{lem}
Operators $\cG_{[1:K]}$ as in \eqref{eq:cGK}
are naturally covered by the main result, Theorem~\ref{thm:MainRes}. 
To apply Theorem~\ref{thm:MainRes} to $\cG_{[1:K]}$, it is in fact sufficient that only the last component $\cG_K$ satisfies 
    Assumption~\ref{ass:lipschitz}\,, in the sense that 
there exist $s > \frac{1}{2}$, $t$ and $r>0$ 
such that 
\begin{equation}\label{eq:lipschitz3}
	\|\cG_K(x)-\cG_K(x')\|_{\cX_K^{{t}}}\le L_k \|x-x'\|_{\cX_{K-1}}, 
	\quad x,x'\in C_{{r}}^{{s}}(\cX_{K-1}).
\end{equation}
For the remaining ``factor'' operators $\cG_k$ we only need to assume the weaker condition 
$\cG_k \in \Lip(\cX_{k-1},\cX_k)$ for $k=1,...,K-1$ with Lipschitz constants $L_k>0$.
Lemma~\ref{lem:LipComp} then implies for $\cY=\cX_K$ and $\cX=\cX_0$ the bound 
\begin{equation}\label{eq:lipschitzBd}
	\|\cG_{[1:K]}(x)-\cG_{[1:K]}(x')\|_{\cY^{{t}}}\le L_{[1:K]} \|x-x'\|_{\cX}, 
	\quad x,x'\in C_{{r}}^{s}(\cX).
\end{equation}

\subsection{H\"older Continuous Maps}
\label{sec:Hoelder}
The proposed methodology readily extends to the case of 
$\gamma$-Hölder continuous operators, 
with some $\gamma \in (0,1]$ 
(note that, unlike Lipschitz-continuous maps,
$\gamma$-H\"older-continuous maps are, for $\gamma \in (0,1)$, 
not closed under composition). 
We briefly indicate the modifications of our main result, Theorem~\ref{thm:MainRes},
to $\gamma$-Hölder continuous operators, and state the corresponding emulation rate bounds.

Suppose then that there is $\gg\in(0,1)$ and $L_\cG>0$
such that Assumption~\ref{ass:lipschitz} holds in the weaker form 
	\begin{equation*}
		\|\cG(x)-\cG(x')\|_{\cY^t}\le L_\cG \|x-x'\|^\gg_\cX, \quad x,x'\in C_r^s(\cX).
	\end{equation*}
	The corresponding growth bound from~\eqref{eq:lin_growth} 
        then translates to
	\begin{equation}\label{eq:hoelder_growth}
		\|\cG(x)\|_{\cY^t}\le C(1+\norm{x}_\cX^{\gg}),
		\quad \mbox{with}\quad  C:=\max(L_\cG, \norm{\cG(0)}_{\cY^t}).
	\end{equation}
This results in the same rate for the output truncation error 
from Proposition~\ref{prop:output_trunc}, 
where $\gg$ only enters the hidden constant. 
On the other hand, the input truncation error from Proposition~\ref{prop:input_trunc} 
now scales worse and is of order 
	\begin{equation*}
		\cO\left(\max_{j=1,\dots,N}  
                 w_j^{t-\frac{1}{2}-\gd} w_{M_j+1}^{\gg(s-\frac{1}{2}-\gd)}\right), 
                 \quad \text{and} \quad 
		\cO\left( w_{M+1}^{\gg(s-\frac{1}{2}-\gd)}\right),
	\end{equation*}
respectively.

In addition, the exponents $\ga,\ol \ga, \gb, \ol \gb$ in Assumption~\ref{ass:surrogate} 
have to be scaled by $\gg^{-1}\in(1,\infty)$ for the DNN emulation of Hölder continuous maps, 
as may be seen e.g.\ from \cite[Corollary 2.2]{shen2022neural} or \cite[Corollary 1.2]{shen2021neural}. 
In all,
the estimates on the number of parameters in Theorem~\ref{thm:MainRes} 
change in the case of Hölder-continuous operators to 
\begin{equation*}
	\cN_{para}\lesssim
		\begin{cases}
			\eps^{-\ga/\gg^2(s-1/2)-(2\gg+\gb)/2\gg t -\gd}	(1+\abs{\log(\eps)})^{\gk/\gg}
			&\quad\text{if $t\le \frac{1}{2}$, } \\
			\eps^{-\ga/\gg^2(s-1/2)-1/t-\gb/\gg -\gd}	(1+\abs{\log(\eps)})^{\gk/\gg}
			,
			&\quad\text{if $t> \frac{1}{2}$.}
		\end{cases}
\end{equation*}
Note that the scaling of $\gg^{-2}$ in first term of the exponent 
is due to the reduced rate for the input truncation error 
\emph{and} 
deteriorating DNN emulation rates for Hölder continuous functions.
The result \cite[Theorem~2.1]{shen2022neural} even allows to obtain
expression rate bounds when only a (possibly weak) bound on the 
modulus of continuity of the components $\mathfrak{g}_j$ of $G$
holds, as e.g. in the so-called Calder\'{o}n-Problem (where this modulus
is logarithmic, see \cite{KochRuland21} and the references there).

\section{Examples}
\label{sec:Examples}
We provide several examples to illustrate the scope of the
presently obtained expression rate bounds.  Naturally, 
all maps between function spaces which are holomorphic, as considered
in \cite{herrmann2022neural} are in particular (locally) Lipschitz
and are therefore covered by the present approximation rate bounds.
In \cite{herrmann2022neural}, under parametric holomorphy assumptions,
operator surrogates were constructed
which were based on strict ReLU feedforward NNs. The present
constructions are considerably more involved than those in \cite{herrmann2022neural},
so that examples from \cite{herrmann2022neural} are not illustrative.
We opt to discuss examples of Lipschitz operators that generically 
do not exhibit any regularity beyond Lipschitz (or H\"older). 

In Section~\ref{sec:VI}, 
we show that the proposed, abstract setting naturally accommodates a broad class of
  (non-linear) parametric, elliptic variational inequalities.  
In Section~\ref{sec:LipHS}, 
we study expression rates for certain Lipschitz maps on the space of Hilbert-Schmidt
operators between two real, separable Hilbert spaces $\cH_1$ and $\cH_2$.
	\subsection{Elliptic Variational Inequalities}
	\label{sec:VI}
       We consider parameter-to-solution maps of elliptic variational inequalities (EVIs), 
       that arise for example from obstacle problems (e.g.\ \cite{DuvautLions,Hlavacek}) 
       in mechanics, optimal stopping in financial modeling or 
       in optimal control of differential operators \cite{TFMKOptCtrl2022}.
       It is well-known that the dependence of solutions of EVIs in unilateral problems for 
       elliptic operators on, for example, 
       coefficient functions in the elliptic operator is Lipschitz, 
       but not better, even for smooth obstacles.
       See, e.g., \cite{doktor1980perturbations,GJBKR} and the references there. 
       Hence, within our presently developed, abstract framework, 
       coefficient-to-solution maps of EVIs can be emulated by Deep Operator Nets (DONs), 
       with the architecture as in~\eqref{eq:DONarch}.
       We remark that the ensuing DON emulation rate analysis 
       is based on (classical) results on the Lipschitz stability
       of solutions of EVIs under perturbations of the data, 
       as proved e.g.\ in \cite{doktor1980perturbations}.
       An alternative approach to DON expression rate bounds 
       is via ``unrolling'' known, iterative solution algorithms 
       e.g.\ with recurrent DNNs, see e.g.\ \cite{DeepSolOpPrxNet}.
\subsubsection{Abstract Setting}
\label{sec:AbsSet}
	Let $(\cH,\inp{\cdot}{\cdot}_\cH )$ be a separable Hilbert space, 
        let $\cH^*$ denote the dual space of $\cH$, 
        and let 
        $\dualpair{\cH^*}{\cH}{\cdot}{\cdot}:\cX^*\times \cX\to \bR$ 
        be the associated dual pairing. 
        Further, let $\cA_1,\cA_2:\cH\to\cH$ be (possibly nonlinear) operators, 
        $\cK_1, \cK_2\subset\cH$ and let $f_1,f_2\in\cH$.
	\begin{assumption}\label{ass:VI}
	It holds that $\cK_1,\cK_2\subset\cH$ are closed convex subsets and 
        there exist constants $\ell, L>0$ such that for $i=1,2$:  
		\begin{itemize}
			\item $\cA_i$ is strongly $\ell$-monotone, i.e.  
			\begin{equation}
				\inp{\cA_iu-\cA_iv}{u-v}_\cH\ge \ell \norm{u-v}_\cH^2\quad \text{for all $u,v\in\cH$.}
			\end{equation}	
			\item $\cA_i$ is $L$-Lipschitz continuous, i.e.
			\begin{equation}
				\norm{\cA_iu-\cA_iv}_\cH\le L \norm{u-v}_\cH\quad \text{for all $u,v\in\cH$.}
                              \end{equation}
                            \item For all $r>0$ exists $B(r) <\infty$ such that
			\begin{equation*}
				\sup_{\|u\|_\cH\le r} \norm{\cA_iu}_\cH\le B(r).
			\end{equation*}
		\end{itemize}
	\end{assumption}

	For ${i=1,2}$ and $\cK_i\neq\emptyset$ closed and convex, we define by 
	\begin{equation}
		P_i:\cH\to\cK_i,\quad u \mapsto \operatorname{argmin}_{v\in \cK_i}\|u-v\|_\cH
	\end{equation}
	the projection onto $\cK_i$.
	We consider the variational inequality to 
	\begin{equation}\label{eq:VI}
		\text{find $u_i\in\cK_i$ such that}\quad  
		\inp{\cA_iu_i}{v-u_i}_\cH \ge \inp{f_i}{v-u_i}_\cH\quad
		\text{for all $v\in K_i$.}
	\end{equation} 
	Under Assumption~\ref{ass:VI}, it is well-known that~\eqref{eq:VI} admits a unique solution $u_i\in \cK_i$. 
	In addition, 
        $u_i$ depends Lipschitz continuous 
        on
        $\cA_i$, $f_i$ and $\cK_i$ in the following sense. 
	\begin{thm}\cite[Theorem 2.1]{doktor1980perturbations}\label{thm:VI-perturb}
		Let Assumption~\ref{ass:VI} hold and define for any $r>0$
		\begin{align*}
			\mfA(r)&:=\sup_{\|u\|_\cH\le r} \norm{\cA_1u-\cA_2u}_\cH<\infty, 
                        \\
			\varrho(r)&:=\sup_{u\in\cH, \|u\|_\cH\le r} \norm{P_1u-P_2u}_\cH<\infty,\quad\text{and} 
                        \\
			\mfD&:=\max_{i=1,2} \inf_{\varphi\in\cK_i}\norm{\varphi}_\cH.
		\end{align*}
		Let $u_i$ be the unique solution to~\eqref{eq:VI} for ${i=1,2}$.
		Then, 
                there exist constants $R=R(\ell , L , \mfD) > 0$ 
                and $C=C(\ell , L) > 0$ such that 
                $\norm{u_i}_\cH\le R$ for $i=1,2$ 
                and
		\begin{equation}\label{eq:Vi-Lipschitz}
			\|u_1-u_2\|_\cH\le C \left(
			\varrho\left(R+B(R)+\max\left(\norm{f_1}_\cH, \norm{f_2}_\cH\right)\right) + 
			\|f_1-f_2\|_\cH + \mfA(R)\right).
		\end{equation}
	\end{thm}
	\subsubsection{Elliptic Variational Inequalities on the Torus}
	\label{sec:EVI}
	%
        For $d\in\N$ denote the $d$-dimensional torus by $\bT^d\simeq[0,1)^d$.
        Let $\cH:=H^1(\bT^d)$ and consider the closed, convex sets $\cK_\phi\subset\cH$, 
        that are
        parameterized by \emph{obstacles} $\phi\in\cH$ via 
	\begin{equation*}
		\cK_\phi:=\{ \varphi\in\cH: \varphi\ge \phi  \}.
	\end{equation*}
	We fix a (possibly non-linear) operator $\cA:\cH\to\cH$ satisfying Assumption~\ref{ass:VI} 
       and a source term $f\in \cH$, and identify $f$ with an element $f^*\in\cH^*$ by the Riesz representation theorem.
	For the sake of brevity, we only consider varying obstacles $\phi\in\cH$ in the following. 
	Then, for any $\phi\in\cH$ and corresponding $\cK_\phi$, 
        there exists a unique solution $u_\phi\in\cK_\phi$ to the variational inequality
	\begin{equation}
		\inp{\cA u_\phi}{v-u_\phi}_\cH \ge \dualpair{\cH^*}{\cH}{f^*}{v-u_\phi}\quad
		\text{for all $v\in\cK_\phi$.}
	\end{equation} 
	We may thus define the \emph{obstacle-to-solution operator}
	\begin{equation}\label{eq:obs-to-sol-operator}
		\cG: \cH \to \cH, \quad  \phi \mapsto u_\phi.
	\end{equation} 
	Now let $\phi_1,\phi_2\in \cH$ and denote the respective projections by $P_i:\cH\to \cK_{\phi_i}$. 
	In view of Theorem~\ref{thm:VI-perturb}, we obtain that
	\begin{equation*}
		\begin{split}
			\norm{\cG(\phi_1)-\cG(\phi_2)}_{\cH}
			&\le C \varrho(R+B(R)+\norm{f^*}_{\cH^*}) \\
			&\le C \sup_{\|u\|_\cH\le R+B(R)+\norm{f^*}_{H^*}} 
			\norm{P_1u-P_2u}_\cH,
		\end{split}
	\end{equation*} 
	where $C>0$ is independent of $\phi_1,\phi_2$, 
        but with $R>0$ depending on $\mfD=\max_{i=1,2} \inf_{\varphi\in\cK_i}\norm{\varphi}_\cH$.
	
	Using that $P_1u=P_2(u-(\phi_1-\phi_2))+(\phi_1-\phi_2)\in\cK_{\phi_1}$, however, 
        shows that there exists a $L_\cG>0$, independent of $R$, 
        such that for all $\phi_1, \phi_2\in\cH$ 
        it holds  
	\begin{equation}
		\begin{split}\label{eq:VI-lipschitz}
			\norm{\cG(\phi_1)-\cG(\phi_2)}_{\cH}
			&\le C \sup_{\|u\|_\cH\le R+B(R)+\norm{f^*}_{H^*}} 
			\norm{P_2(u-(\phi_1-\phi_2))+(\phi_1-\phi_2)-P_2u }_\cH\\
			&\le C \sup_{\|u\|_\cH\le R+B(R)+\norm{f^*}_{H^*}} 
			\norm{P_2(u-(\phi_1-\phi_2))-P_2u}_\cH+\norm{\phi_1-\phi_2}_\cH\\
			&\le L_\cG \norm{\phi_1-\phi_2}_\cH.
		\end{split}
	\end{equation} 
	\subsubsection{Wavelet Encoding}
	\label{sec:wavelets}
	%
	Let $K_j:=\{k\in\bZ^d:\;0\le k_1,\dots,k_d<2^{j}\}\subset 2^{j}\bT^d$ for $j\in\bN_0$, 
       let $\cL_0:=\{0,1\}^d$ and $\cL_j:=\cL_0\setminus \{(0,\dots,0)\}$ for $j\in\bN$.
	By \cite[Proposition 1.34]{TriebelWavDom}, 
        there exists an $L^2(\bT^d)$-orthonormal basis
	\begin{equation}\label{eq:torusbasis2}
		\mathbf\Psi:=\left((\psi_{j,k}^l),\;(j,k,l)\in \cI_{\mathbf \Psi} \right),
		\quad
		\cI_{\mathbf \Psi}:=\{j\in\bN_0,\; k\in K_j,\;l\in\cL_j\},
	\end{equation}
	where the $\psi_{j,k}^l:\bT^d\to \bR$ 
        are constructed from scaled, translated and tensorized one-periodic wavelets,
        such that $\gl(\text{supp}(\psi_{j,k}^l)) = \cO(2^{-dj})$.
	The basis $\mathbf\Psi$ may be constructed from Daubechies wavelets with 
        $m\in\bN$ vanishing moments that have compactly supported, univariate scaling and wavelet functions. 
	By choosing $m$ sufficiently large, 
        we may ensure for given $k\in\bN$ that $\psi, \phi\in \rC^k(\bR)$, 
        and thus $\mathbf \Psi\subset \rC^k(\bR)$. 
        For instance, $\psi, \phi\in \rC^1(\bR)$ holds for 
        univariate so-called Daubechies wavelets with $m \ge 5$ vanishing moments, 
        see~\cite[Section 7.1]{daubechies1992ten}.
	
	For $\gg\in(0,k)$, $p\in[1,\infty]$ and $\varphi\in L^2(\bT^d)$ 
       recall from \cite{TriebelWavDom} the \emph{Besov norms} 
	\begin{equation}\label{eq:besovnorm}
		\|\varphi\|_{B_{p,p}^\gg(\bT^d)}:=
		\left(
		\sum_{(j,k,l)\in\cI_{\mathbf\Psi}} 2^{jp(\gg+\frac{d}{2}-\frac{d}{p})} |(\varphi , \psi_{j,k}^l)_{L^2(\bT^d)}|^p
		\right)^{1/p},
		\quad
		p\in[1,\infty),
	\end{equation}
	and, for $p=\infty$,
	\begin{equation}\label{eq:besovnorminf}
		\|\varphi\|_{B_{\infty,\infty}^\gg(\bT^d)}:=
		\sup_{(j,k,l)\in\cI_{\mathbf\Psi}} 2^{j(\gg+\frac{d}{2})} |(\varphi ,\psi_{j,k}^l)_{L^2(\bT^d)}|
		<\infty.
	\end{equation}
	According to \cite[Theorem 1.36]{TriebelWavDom}, 
        the corresponding one-periodic \emph{Besov spaces} on $\bT^d$ are then represented 
        via 
	\begin{equation}\label{eq:besovspace}
		B_{p,p}^\gg(\bT^d):=\set{\varphi\in L^2(\bT^d):\; \|\varphi\|_{B_{p,p}^\gg(\bT^d)}<\infty}.
	\end{equation}
	We recall that $B_{2,2}^\gg(\bT^d)=H^\gg(\bT^d)$. 
	Further, let $\cC^\tau(\bT^d)$ denote the Hölder-Zygmund space for a given exponent $\tau>0$. 
	There holds the embedding 
        $B_{p,p}^\gg(\bT^d)\hookrightarrow B_{\infty,\infty}^\tau(\bT^d) 
        =
        \cC^\tau(\bT^d)$ for $\gg-\frac{d}{p}>0$ and $\tau\in\left(0,\gg-\frac{d}{p}\right]$, 
        \emph{with embedding constant bounded by one}, see e.g.\ \cite[Chapter 2.1]{TriebelTOFS4}.
	
	Since $B_{2,2}^\gg(\bT^d)=H^\gg(\bT^d)$ is a Hilbert space, 
        for representations in Fourier- or Wavelet-bases 
        one may identify  $B_{2,2}^\gg(\bT^d)$ with certain smoothness spaces $\cX^s$ and $\cY^t$ 
        as in Section~\ref{sec:smoothness_scales}. 
        To relate the exponent $\gg$ with $s$ and $t$, 
        we derive an equivalent norm to~\eqref{eq:besovnorm} for $p=2$.
        It is based on a weight sequence $\bw=(w_i,i\in\bN)$ 
        with a single integer index as in Section~\ref{sec:smoothness_scales}.
	%
        As a first step, observe that $|K_j|= 2^{dj}$ and $|\cL_j|\le 2^d$ for all $j\in\bN_0$, 
        and denote by $(i_j,j\in\bN_0)$ an (arbitrary) collection of bijective mappings 
        satisfying
	\begin{align*}
		&i_0:K_0\times \cL_0\to \{0, \dots, 2^d-1\} \\
		&i_j:K_j\times \cL_j\to 
		\left\{\sum_{m=0}^{j-1} |K_m||\cL_m| + 1, \dots, \sum_{m=0}^j |K_m||\cL_m| \right\},\quad j\in\bN.
	\end{align*} 
	We may then re-label all wavelet indices $(j,k,l)$ by integers via the one-to-one mapping
	\begin{equation*}
		\mfI:\mathbf\cI_\Psi\to \bN_0,\quad (j,k,l)\mapsto i_j(k,l).
	\end{equation*}
	Note that 
	\begin{equation*}
		\sum_{m=0}^{j-1} |K_m||\cL_m|=1+(2^d-1)\sum_{m=0}^{j-1} d^{dm}=2^{dj}, \quad j\in\bN,
	\end{equation*}
	hence $\mfI(j,k,l)\in\{2^{dj}+1, \dots, 2^{d(j+1)}\}$ for any $(j,k,l)\in\mathbf\cI_\Psi$ with $j\ge1$.
	Thus, by letting $\psi_{\cI(j,k,l)}:=\psi_{j,k}^l$, there holds  
	\begin{equation}\label{eq:norm_equiv2}
		\begin{split}
			\|\varphi\|_{B_{p,p}^\gg(\bT^d)}^p
			&=
			\sum_{(j,k,l)\in\cI_{\mathbf\Psi}} 2^{jp(\gg+\frac{d}{2}-\frac{d}{p})} |(\varphi 	,\psi_{j,k}^l)_{L^2(\bT^d)}|^p\\
			&\le 
			\sum_{i\in\bN_0} i^{\frac{p\gg}{d}+\frac{p}{2}-1} |(\varphi ,\psi_i)_{L^2(\bT^d)}|^p \\
			&\le 2^d 
			\sum_{(j,k,l)\in\cI_{\mathbf\Psi}} 2^{jp(\gg+\frac{d}{2}-\frac{d}{p})} |(\varphi 	,\psi_{j,k}^l)_{L^2(\bT^d)}|^p\\
			&= 2^d \|\varphi\|_{B_{p,p}^\gg(\bT^d)}^p.
		\end{split}
	\end{equation}
	Now define the Hilbert space $(\cX, \inp{\cdot}{\cdot}_\cX)$ with the inner product 
	\begin{equation*}
		\inp{\varphi}{ \phi}_\cX
		:=\sum_{i\in\bN_0}\inp{\varphi}{i^{\gg/2d}\psi_i}_{L^2(\bT^d)}
		\inp{\phi}{i^{\gg/2d}\psi_i}_{L^2(\bT^d)},
		\quad \varphi, \phi\in L^2(\bT^d)
	\end{equation*}
	and by
	\begin{equation*}
		\cX:=\set{\varphi\in L^2(\bT^d):\; \norm{\varphi}_\cX:=\sqrt{\inp{\varphi}{\varphi}_\cX}<\infty}.
	\end{equation*}
	With $p=2$ in \eqref{eq:norm_equiv2} there holds  
        $\cX=H^\gg(\bT^d)$ with the norm equivalence
	\begin{align*}
		\norm{\varphi}_{H^\gg(\bT^d)}^2
		\le 
		\norm{\varphi}_{\cX}^2
		\le 
		2^d \norm{\varphi}_{H^\gg(\bT^d)}^2.
	\end{align*}
	\subsubsection{Neural Network Approximation Rates of parametric EVIs}
	\label{sec:VI-approx}
	Let $w_i:=i^{-1}, i\in\bN$.
        Then $\bw := (w_i,i\in\bN)\in \ell^{1+\eps_\bw}(\bN)$ for all $\eps_\bw>0$.
        Choose $\cX:=H^{s_0}(\bT^d)$ for some fixed $s_0\ge0$, 
        to obtain for any $s\ge 0$ and $p=2$ that
	\begin{align*}
		\cX^s&=\left\{ \varphi\in\cX:\; 
                    \|\varphi\|^2_{\cX^s}
                    := \sum_{i\in\bN} |(\varphi ,\psi_i)_\cX|^2w_i^{-2s} <\infty\right\}
                 \\
		&=\left\{ \varphi\in H^{s_0}(\bT^d):\; \|\varphi\|^2_{\cX^s}
                  := \sum_{i\in\bN} |(\varphi ,\psi_i)_{H^{s_0}(\bT^d)}|^2w_i^{-2s} <\infty\right\}
                 \\
		&=\left\{ \varphi\in L^2(\bT^d):\; \sum_{i\in\bN} i^{2(\frac{s_0}{d}+s)}|(\varphi ,\psi_i)_{L^2(\bT^d)}|^2 < \infty\right\} 
                \\
		&= 
		H^{s_0+ds}(\bT^d) 
                \\
        &= 
                B^{s_0+ds}_{2,2}(\bT^d).
	\end{align*}
	Similarly, with $\cY:=L^2(\bT^d)$, 
        it follows that $\cY^t=H^{dt}(\bT^d) = B^{dt}_{2,2}(\bT^d)$ for any $t\ge0$.
	
	Now we fix $s_0:=1$, hence $\cX=H^{1}(\bT^d)$, and let $s>\frac{1}{2}$. 
	For any $r>0$ and $\phi\in C_s^r(\cX)$, 
        there holds by~\eqref{eq:ball-cube-relation} that 
        $\phi\in \cX^{s-\frac{1}{2}-\eps} = H^{1+d(s-\frac{1}{2}-\eps_0)}(\bT^d)$ 
        for any $\eps_0\in (0, s-\frac{1}{2})$. 
        On the other hand, 
        $\phi\in B_r(\cX^s)=B_r(H^{s_0+ds}(\bT^d))$ 
        is sufficient to ensure $\phi\in C_s^r(\cX)$.
	Assumption~\ref{ass:lipschitz} 
        is thus satisfied for any $s>\frac{1}{2}, r>0$ and $t=\frac{1}{d}$, 
        since for $\phi_1,\phi_2\in C_s^r(\cX)$ we have by~\eqref{eq:VI-lipschitz} 
        that
	\begin{equation*}
		\|\cG(\phi_1)-\cG(\phi_2)\|_{\cY^t}\le L_\cG \|\phi_1-\phi_2\|_{\cX}.
	\end{equation*}
	We are now in a position to bound the mean-squared error for 
        obstacles $\phi\in\cX$ of the form $\phi=\gs_r^s(\mfu)$ with $\mfu\in U=[-1,1]^\bN$, 
	given as realizations of the $\cX$-valued random variable 
	\begin{equation*}
		\gs_r^s:U\to \cX, \quad \mfu\mapsto r\sum_{i\in\bN} w_i^s\mfu_i\psi_i.
	\end{equation*}

	Provided that Assumption~\ref{ass:surrogate} holds for fixed exponents 
        $\ga\ge 1$ and $\ol\ga, \gb,\ol\gb\ge 0$, 
        Theorem~\ref{thm:MainRes} shows that for any $\eps\in(0,1]$, 
        there exists a finite-parametric 
        neural network approximation with 
        at most $\cN_{para}(\eps) \in \bN$ parameters 
        to $\cG$ 
        such that it holds
	\begin{equation*}
		\|\cG-\widetilde\cG\|_{L^2(C_r^s(\cX), (\gs_r^s)_\#\cP_U; \cY)} \le \eps,
	\end{equation*}
	and for any $\gd>0$ 
        there exists a constant $C>0$ depending on $\delta$ 
        such that 
 	\begin{equation*}
		\cN_{para}\le
		C
		\begin{cases}
			\eps^{-\ga/(s-1/2)-d(1+\gb/2)-\delta}(1+\abs{\log(\eps)})^\gk, 
			&\quad d\ge 2, \\
			\eps^{-\ga/(s-1/2)-1-\gb-\delta}(1+\abs{\log(\eps)})^\gk, 
			&\quad d=1.
		\end{cases}
	\end{equation*}
	\begin{rem}
		In general $\widetilde\cG(\phi)\notin\cK_\phi$ 
        for a given $\phi\in C_s^r(\cX)$ due to the truncation and 
        neural network approximation of the coordinate mappings $\mfg_j$ in~\eqref{eq:coordinate_g_j}. 
        Further, let $\phi_M:=\sum_{i=1}^M \cE_\cX(\phi)_i\psi_i$ denote the 
        $M$-term approximation of $\phi$ for fixed $M\in\bN$, and 
        assume that $M$ is fixed for all dimensions of the truncated output for simplicity (cf. Section~\ref{sec:input_trunc}). 
        Then, in general also $\widetilde\cG(\phi)\notin\cK_{\phi_M}$, 
        due to the bias from output truncation and the neural network surrogates.
		
	However, a finite dimensional approximation $\widetilde\cG^+(\phi)\in\cK_{\phi_M}$ 
        may be achieved by the following post-processing step.
        Let 
		\begin{equation}
			\widetilde\cG^+(\phi):\cX \to \cY,\quad \phi\mapsto \max\left(\widetilde\cG(\phi), 
                          \sum_{i=1}^M \cE_\cX(\phi)_i\psi_i\right).
		\end{equation}
	The maximum is understood in the point-wise sense. It is well-defined,
    since we assumed at hand a continuous wavelet basis of $L^2(\bT^d)$ for de- and encoding.
		
       Augmenting  $\widetilde\cG(\phi)$ by the $M$-term truncation of $\phi$ involves 
       $M=\cO(\eps^{-1/(s-\frac{1}{2}-\gd)})$ additional parameters (cf.~\eqref{eq:M} 
       in the proof of Theorem~\ref{thm:MainRes}), 
       and therefore does not dominate the asymptotic complexity.
       Furthermore,
		\begin{align*}
			\norm{\widetilde\cG^+(\phi)-\cG(\phi)}_\cY
			&\le \norm{\widetilde\cG^+(\phi)-\max(\widetilde\cG(\phi), \phi)}_\cY 
                           + \norm{\max(\widetilde\cG(\phi), \phi) -\max(\cG(\phi), \phi)}_\cY
                        \\
			&\le \norm{\phi_M-\phi}_\cY + \norm{\widetilde\cG(\phi) -\cG(\phi)}_\cY, 
		\end{align*}
		and since $\cX\hookrightarrow \cY$ there holds 
		\begin{equation*}
			\|\cG(\phi)-\widetilde\cG^+(\phi)\|_{L^2(C_r^s(\cX), (\gs_r^s)_\#\cP_U; \cY)} \le C \eps,
		\end{equation*}
		for a $C>0$, independent of $\eps$.
	\end{rem}
\subsection{Expression Rates for Lipschitz Maps of Hilbert-Schmidt Operators}
\label{sec:LipHS}
Our results also encompass the approximation of nonlinear
Lipschitz maps between spaces of operators. 
We illustrate this for
the particular class of Hilbert-Schmidt (HS) operators,
$\cS_2(\cH_1,\cH_2)$ (``$2$-Schatten'' class of linear operators),
acting between separable Hilbert spaces $\cH_1$ and $\cH_2$.  As
$\cS_2(\cH_1,\cH_2)$ is itself a Hilbert space, the presently
developed abstract framework \eqref{eq:cG} is applicable with
$\cX = \cY = \cS_2(\cH_1,\cH_2)$.  Assuming at hand orthonormal basis
$(\psi^{\cH_i}_j, j\in\bN)$ of $\cH_i$ for $i=1,2$, adopting the
dyadic basis $(\psi^{\cH_1}_j\otimes \psi^{\cH_2}_{j'}, j,j'\in\bN )$
of $\cX = \cS_2(\cH_1,\cH_2)$, the spaces $\cX^s$ correspond to
$p$-Schatten classes for suitable $p(s)<2$.

We first present a general setting 
without particular structural assumptions on the map $\cG$
and then, in Sect.~\ref{sec:FuncCalc}, 
address a particular case of ``singular value maps'', 
through a scalar Lipschitz function, 
of HS operators as considered in \cite{andersson2016operator}.
\subsubsection{Hilbert-Schmidt Operators}
\label{sec:GenLipHS}
We denote by $(\cH_i, (\cdot,\cdot)_{\cH_i})$, $i=1,2$
separable Hilbert spaces and assume that $A:\cH_1\to \cH_2$ is a compact linear operator.
Then there exists a singular value decomposition (SVD) of $A$, i.e., 
there is a sequence $ \{ s_j(A): j \in \bN\} \subset [0,\infty)$, 
and
ONB $(v_i, i\in\bN)$ of $\cH_1$
and another 
ONB $(w_j, j\in\bN)$ of $\cH_2$ 
such that 
\be\label{eq:SVD}
A = \sum_{j\in\bN} s_j(A) w_j \otimes v_j \;,
\ee
where $w_j \otimes v_j\in\cL(\cH_1, \cH_2)$ 
is defined by 
$(w_j \otimes v_j) \phi = \inp{\phi}{v_j}_{\cH_1}w_j$. 
The real, non-negative numbers 
$s_j(A)\ge0$ in \eqref{eq:SVD} are the \emph{singular values of $A$}.
They accumulate only at zero.
For $0<p\leq \infty$, 
we denote the subset of \emph{$p$-Schatten class operators} 
as
$$
S_p(\cH_1,\cH_2) = \{ A \in \cL(\cH_1,\cH_2): (s_j(A), j\in\bN) \in \ell^p(\bN) \}
\;.
$$
Of particular interest in the present context is the case $p=2$, 
the so-called Hilbert-Schmidt (HS) operators, with square-summable singular values 
$(s_j(A), j\in\bN) \in \ell^2(\bN)$.
For $A\in \cS_2(\cH_1,\cH_2)$, 
the sum in~\eqref{eq:SVD} converges in $\cL(\cH_1,\cH_2)$.
We recall that $\cS_2(\cH_1,\cH_2)$ is a separable Hilbert space with basis 
$(w_j\otimes v_i,\, i,j\in\bN)$ and inner product given by
$$
(A,B)_{HS} := \sum_{j\in\bN} (Ae_j,Be_j)_{\cH_2} \;,
$$
where $(e_k, k\in\bN) \subset H_1$ denotes an orthonormal basis~\footnote{The inner product 
and norm are independent of the choice of orthonormal basis.} 
of $\cH_1$, 
and $(\cdot,\cdot)_{\cH_2}$ denotes the $\cH_2$ inner product. 
A corresponding norm in $S_2(\cH_1 , \cH_2)$ is given by 
$$
\| A \|_{\cS_2(\cH_1,\cH_2)} 
:= 
(A,A)_{HS}^{1/2} \;,\quad A \in \cS_2(\cH_1,\cH_2)\;.
$$
\subsubsection{Singular Value Lipschitz Functional Calculus}
\label{sec:FuncCalc}
A particular class of nonlinear maps on the space of HS operators 
can be constructed via \emph{functional calculus}. 
For such maps, 
better DON emulation rates can be shown than in the general case.
Let $f:\bR_{\ge0}\to \bR$
\footnote{It is also possible to define $\cG_f$ in terms of a complex-valued, scalar Lipschitz function $f:\bR_{\ge0}\to\bC$. 
For the ensuing DNN emulation of operator $\cG_f$, 
one then has to rely on \emph{complex-valued neural networks} (CVNNs), 
that involve complex activation functions and linear transforms, 
see e.g.\ \cite{caragea2022quantitative}. 
We further note that the right hand side of~\eqref{eq:FCLipschitz} 
only yields the bound $\sqrt{2}L_f$ for complex $f:\bR_{\ge0}\to\bC$.} 
be continuous with $f(0)=0$ and consider the map 
\be\label{eq:FCoperators}
	\cG(f):\cS_2(\cH_1,\cH_2)\to\cL(\cH_1,\cH_2), \quad
	A \mapsto \sum_{j\in\bN} f(s_j(A))\,w_j \otimes v_j.
\ee
Note that we may have 
$\cG(f)(A)\notin\cS_2(\cH_1,\cH_2)$ for $A\in \cS_2(\cH_1,\cH_2)$, 
without any further assumptions on $f$.
On the other hand, 
if $f$ is Lipschitz with Lipschitz constant $L_f>0$, 
\cite[Theorem 4.2]{andersson2016operator} shows $\cG(f):\cS_2(\cH_1,\cH_2)\to\cS_2(\cH_1,\cH_2)$ 
is Lipschitz with 
\be\label{eq:FCLipschitz}
	L_{\cG(f)}:=\sup_{A,B\in \cS_2(\cH_1,\cH_2), A\neq B}
	\frac{\norm{\cG(f)(A)-\cG(f)(B)}_{\cS_2(\cH_1,\cH_2)}}{\norm{A-B}_{\cS_2(\cH_1,\cH_2)}}
	\le L_f < \infty.
\ee
Now let $\bw = (w_j, j\in\bN)$ be again a given sequence of 
positive weights, so that $\bw \in\ell^{1+\eps}(\bN)$ for all $\eps>0$.
In view of Section~\ref{sec:dim_trunc}, 
we encode elements $A\in\cX=\cS_2(\cH_1,\cH_2)$
using the basis $(w_j\otimes v_i, i,j\in\bN)$ 
and define 
\be
\cX^s:=\set{A\in \cX: \; \norm{A}_{\cX^s}^2 := \sum_{j\in\bN} s_j(A)^2w_j^{-2s}<\infty}, \quad s>0, 
\ee
as well as the cubes
\be
C_r^s(\cX) := \set{A\in \cX: \; \sup_{j\in\bN} s_j(A)^2w_j^{-2s}\le r}, \quad r,s>0.
\ee
In case there are $p\in(0,2)$ and $\gd>0$ such that 
$s_j(A)\le C w_j^{(1+\gd)/p}$ for all $j\in\bN$, 
it follows that $A\in \cS_p(\cH_1,\cH_2)$. 
In addition, 
for all $s\in [0,(1+\gd)(1/p-1/2)]$ there holds $A\in\cX^s$, 
since 
\begin{equation*}
	\norm{A}_{\cX^s}^2= \sum_{j\in\bN} s_j(A)^2w_j^{-2s}
	\le \left(\sup_{j\in\bN} s_j(A)^{2-p}w_j^{-2s}\right)\sum_{j\in\bN} s_j(A)^p
	\le C \left(\sup_{j\in\bN} s_j(A)^{2-p}w_j^{-2s}\right)<\infty.
\end{equation*}
\subsubsection{DNN Emulation Rates}
\label{sec:NNFuncCalc}
We verify that 
operators $\cG(f)$ of the form~\eqref{eq:FCoperators} with scalar, 
nonnegative Lipschitz $f$ are a particular case of our framework. 
Specifically,
the representation~\eqref{eq:FCoperators} implies 
favourable DNN expression rate bounds in terms of 
the number $\cN_{para}$ of neurons.
To verify this, we apply our abstract setting with the choices
$\cX=\cY: = \cS_2(\cH_1,\cH_2)$.
Our construction of a DNN surrogate starts from the 
$N$-term truncated operator $\cG_N(f)$, 
which is given via
\be\label{eq:FCoperators_trunc}
\cG_{N}(f):\cX\to\cY, \quad
A \mapsto \sum_{j=1}^N f(s(A_j))\,w_j \otimes v_j,\quad N\in\bN.
\ee
Note that we have fixed input and output truncation by $N$ terms simultaneously. 
As $f$ is Lipschitz with $f(0)=0$, 
it follows readily that $f(s_j(A))\le L_f s_j(A)$ 
and thus 
\be\label{eq:FC_estimate}
	\norm{\cG(f)(A)-\cG_N(f)(A)}_\cY^2=\sum_{j\ge N} f(s_j(A))^2
	\le L_f^2 \sum_{j\ge N} s_j(A)^2
	.
\ee

Equation~\eqref{eq:FC_estimate} yields 
for any $s>\frac{1}{2}$ and $\gd\in(0,s-\frac{1}{2})$ 
that there is a $C>0$ (depending on $s$ and on $\gd$)
such that
\bee
\norm{\cG(f)(A)-\cG_N(f)(A)}_\cY
\le C w_N^{s-\frac{1}{2}-\gd},
\quad A\in C_r^s(\cX).
\eee

As $\lim_{j\to\infty}s_j(A)=0$, 
we may assume without loss of generality that $s_j(A)\le 1$ for all $j\in\bN$. 
Hence, we only need to replace the univariate Lipschitz mapping $f$ a total of 
$N$ times by a NN surrogate $\widetilde f:[0,1]\to\bR$ 
such that 
$$
\norm{f-\widetilde f\,}_{L^2([0,1])}\le L_f \eps,\quad \eps\in(0,1].
$$
By Assumption~\ref{ass:surrogate},
this may be achieved by \emph{one common scalar surrogate} $\widetilde f$ 
with $\cO(\eps^{-\gb}\abs{\log(\eps)}^{\bar\gb})$ parameters.
Thus, for fixed, scalar Lipschitz $f$ and any $\eps\in(0,1]$, 
there exists a DNN approximation $\widetilde\cG(f)$ with $\cN_{para}\in\bN$ parameters to $\cG(f)$, 
such that 
\begin{equation*}
	\|\cG(f)-\widetilde\cG(f)\|_{L^2(C_r^s(\cX), (\gs_r^s)_\#\cP_U; \cY)} \le \eps,
\end{equation*}
and
for any $\gd>0$ there is $C_\gd> 0$ such that for all $\eps\in(0,1]$ 
holds
\begin{equation*}
	\cN_{para} \le C_\gd \eps^{-1/(s-1/2)-\gb-\delta} (1+\abs{\log(\eps)})^\gk.
\end{equation*}
\section{Conclusions}
\label{sec:Concl}
%
We obtained expression rate bounds for a class of deep operator
networks (DONs) with the architecture \eqref{eq:DONarch} to emulate
Lipschitz continuous maps between separable Hilbert spaces.  It is
based on linear encoder/decoder pairs $(\cE,\cD)$ based on stable
biorthogonal bases of domain $\cX$ and target space $\cY$.  For
example, \KL expansions as employed in so-called PCA-Nets
\cite{lanthaler2023operator}, Fourier bases as used in FNOs
\cite{FNOGeo,UFNo}, etc.  Concrete encoders in this framework either
access PDE inputs via point values as e.g. in ONets
(e.g. \cite{deepONet}) (with dual bases consisting of Dirac measures),
or also via ``Galerkin moments'' as recently promoted in discussion of
transformer-encoders in \cite[Section~4.1.3]{cao2021choose}.  In
either of these cases, \emph{aliasing errors} due to sampling or
quadrature must be accounted for, in addition to the expression error
analysis performed here.  To accommodate possibly low regularity of the
map $\cG$, we considered neural approximators $\tilde{G}$ in
\eqref{eq:DONarch} from two superexpressive classes of DNNs that are
not subject to the CoD: (i) NNs with superexpressive activation
functions and (ii) NNs with nonstandard architecture.  Specifically,
the ``Nest-Net'' construction from \cite{shen2022neural}.

Extensions of the presently developed analysis 
to separable Banach spaces that admit 
stable, biorthogonal representation systems
such as those developed for a broad range of Besov-Triebel-Lizorkin spaces
in domains as constructed e.g.\ 
in \cite{TriebelBases} and the references there are conceivable,
under conditions on these spaces.
We refer to \cite[Section~9, App.~B]{NeurOp} for 
conditions and techniques in such more general settings.
These references considered DONs emulating maps $\cG$ 
between spaces of functions in Euclidean domains.
More general settings in the abstract framework 
of Section~\ref{sec:LipHS} may accommodate DONs that emulate maps from 
certain linear operators between separable Hilbert spaces 
to functions, 
a task that frequently arises in inverse problems,
for example (with suitably regularizing observation functionals).

Also covered are multiresolution encoders and decoders, 
such as wavelets \cite{TriebelWavDom}.
Since multiresolution bases hierarchically encode increasing spatial and temporal 
resolution, the presently developed results and framework comprise also 
so-called ``multi-fidelity'' operator networks as put forward in \cite{MultiFidOnets}.

For Lipschitz continuous $\cG$, in
  \cite{LMK22,lanthaler2023operator} \emph{lower expression rate
    bounds} for DON surrogates 
  were obtained, for approximators $\widetilde{G}$ in
  \eqref{eq:DONarch} being standard feedforward NNs with smooth,
  nonpolynomial activation. 
  Moreover, in the recent work~\cite{lanthaler2023curse}, it was
  demonstrated that emulating Lipschitz operators (without additional
  structural properties) on infinite-dimensional hypercubes using a
  broad class of ONet architectures is subject to the CoD.  See for
  instance the exponential lower bounds on the number of parameters in
  the ONet 
  presented in \cite[Theorem
  2.15]{lanthaler2023curse}. Such statements are not contradictory to
  our results: the analysis in \cite{lanthaler2023curse} specifically
  considers \emph{feedforward ReLU-NN surrogates for the coordinate
    maps and decoders}.
In the present manuscript,
higher
rates were obtained by more general approximators $\widetilde{G}$ in \eqref{eq:DONarch}
with either nonstandard architecture or nonstandard,
``superexpressive'' activations.

The encoder/decoder pairs in \eqref{eq:DONarch} pass between the in-
and output spaces $\cX$ and $\cY$ and the sequence space
$\ell^{2}(\N)$ via linear transformations.  Approximation error bounds
due to finitely truncating coefficient sequences furnished by
\emph{linear encoding} on $\cX^s$ were obtained by $N$-term sequence
truncation.  As is well-known, however, \emph{adaptive, nonlinear
  encoding} could yield the same rates for considerably larger classes
of inputs, from (quasi-)Banach spaces $B^\gamma_{p,p}$ for some
$0<p<2$ \cite{CDDTree2001}.  The extension of the presently proposed
framework to such encoders will be considered elsewhere.

The expression rate bounds for DONs \eqref{eq:DONarch} obtained here are
a consequence of \emph{regularity and sparsity of inputs and outputs}, 
as expressed by weighted sequence summability, 
and the assumed mapping properties of $\cG:\cX^s \to \cY^t$, 
and super expressivity of activations in the DNN emulations 
$\widetilde{\mathfrak{g}}_j$ 
of the components 
$\mathfrak{g}_j$ of $G$ in \eqref{eq:DONarch}.
They cover a large class of operator network constructions.
Architectures which are essentially different
such as U-Nets and transformer-based emulators 
to build $\widetilde{G}$ in \eqref{eq:DONarch}, 
as proposed e.g. in \cite{li2022transformer,cao2021choose} and 
in the references there, or the 
branch-trunk architecture of deepONets \cite{deepONet,MS21_984}
could be investigated similarly. 

We remark that the linear decoding used in the ONet architecture
is necessarily restrictive. 
Generally, images of subspaces of $N$-term truncated 
encoded inputs under $\cG$ become 
\emph{Lipschitz manifolds} embedded into $\cY$. 
\emph{Linear decoding} as assumed in Section~\ref{sec:output_trunc} 
of the present analysis will in general poorly capture such 
manifolds, without any further assumptions. 
Here, we \emph{assumed} $\cG(\cX^s)$ to be contained in $\cY^t$.
This assumption is often satisfied in data-to-solution maps for 
elliptic and parabolic PDEs, when $\cX^s$ and $\cY^t$ coincide with
suitable function spaces of Sobolev, resp. of Besov-Triebel-Lizorkin type.
In such spaces, the (linear) $M$-term decoding in Prop.~\ref{prop:input_trunc}
provides corresponding approximation rates. 
For the more general case of the range of $\cG$ on $N$-term encoded inputs
in $\cX$, \emph{nonlinear decoding} as in the branch-trunk architecture of \cite{ChenChen1993},
which allows for data-dependent representation system in the output-decoder,
could result in better expressivity, closer to the \emph{Lipschitz-width benchmark}
\cite{petrova2023lipschitz}.

     \appendix
\section{Proof of Universality}
\label{app:ProofUni}

  \begin{lem}\label{lem:S}
    Let $E \subseteq \ell^2(\N)$ be compact. 
    Then
    \begin{equation*}
      S:= E \cup\{(c_1,\dots,c_n,0,0,\dots) : \bc\in E,~n\in\N\}
    \end{equation*}
    is a compact subset of $\ell^2(\N)$.
  \end{lem}
  \begin{proof}
    Let $(O_i)_{i\in I}$ be an open cover of $S$. 
    We need to show there exists a finite subcover.

    For each $\bc=(c_j)_{j\in\N}\in E$ there exists $\eps_\bc>0$ and $i_\bc\in I$
    such that
    \begin{equation}\label{eq:ic}
      B_{\eps_\bc}(\bc) := \set{\bd\in\ell^2(\N):\norm{\bc-\bd}_{\ell^2}<\eps_\bc}\subseteq O_{i_\bc}.
    \end{equation}
    Additionally, there exists $n_\bc\in\N$ such that
    $\sum_{j>n_\bc} c_j^2< \left({\eps_\bc}/{3}\right)^2$. 
    Then for any
    $\bd\in B_{\eps_\bc/3}(\bc)$ and any $n > n_\bc$ there holds
    \begin{align}\label{eq:norm3}
      \norm{\bc - (d_1,\dots,d_n,0,0,\dots)}_{\ell^2}&\le
      \norm{\bc-\bd}_{\ell^2}+\Bigg(\sum_{j>n_\bc}(d_j^2-c_j^2)+\sum_{j>n_\bc} c_j^2\Bigg)^{1/2}\nonumber\\
      &\le \frac{\eps_\bc}{3} + \Big(\frac{\eps_\bc^2}{3^2} + \frac{\eps_\bc^2}{3^2}\Big)^{1/2}<\eps_\bc.
    \end{align}

    Since $E$ is compact and
    $E\subseteq \bigcup_{\bc\in E}B_{\eps_\bc/3}(\bc)$, 
    there exists
    $m\in\N$ and $\bc_1,\dots,\bc_m\in E$ such that
    $E\subseteq \bigcup_{i=1}^mB_{\eps_{\bc_i}/3}(\bc_i)$.  
    Let $\bd\in E$ be arbitrary. 
    Then there exists $j\in\{1,\dots,m\}$ such
    that $\bd\in B_{\eps_{\bc_j}/3}(\bc_j)$. 
    Thus by \eqref{eq:norm3}
    and \eqref{eq:ic} it holds for all
    $n>N:=\max_{i=1,\dots,m}n_{\bc_i}$
    \begin{equation*}
      (d_1,\dots,d_n,0,0,\dots)\in B_{\eps_{\bc_j}}(\bc_j)\subseteq O_{i_{\bc_j}}.
    \end{equation*}
    This shows
    \begin{equation*}
      E\cup
      \{(d_1,\dots,d_n,0,0,\dots) : \bd\in E,~n>N\}\subseteq \bigcup_{j=1}^m O_{i_{\bc_j}}.
    \end{equation*}

    Finally observe that
    $\bd\mapsto (d_1,\dots,d_n,0,0,\dots):\ell^2(\N)\to\ell^2(\N)$ is
    continuous for any fixed $n\in\N$. Hence compactness of $E$ gives
    compactness of the set of remaining elements
    \begin{equation*}
      R:=\bigcup_{n=1}^N \{(d_1,\dots,d_n,0,0,\dots) : \bd\in E\}.
    \end{equation*}
    Thus $R$ can be covered by finitely many $O_i$,
    and therefore the same holds for $S$.
  \end{proof}

\begin{proof}[Proof of Theorem~\ref{thm:Universality}]
  It suffices to construct $\widetilde\cG_n$, $n\in\N$, such that for
  any compact $K\subseteq\cX$ holds
  $\lim_{n\to\infty}\widetilde\cG_n|_K=\cG|_K$ with uniform
  convergence. 
  Throughout the rest of this proof fix
    $K\subseteq\cX$ compact. We proceed in four steps to construct
    $\widetilde\cG_n$ (independent of $K$) as claimed.
  
  {\bf Step 1.} 
  We claim that for
  every $\delta>0$ exists ${N_1}(\delta,K)\in\N$
  such that
  \begin{equation}\label{eq:claim_m1}
    \sup_{x\in K}\norm{x-\sum_{j=1}^{m}\inp{x}{\widetilde\psi_j}_\cX\psi_j}_\cX
     <
     \delta\qquad\qquad\forall m\ge N_1.
  \end{equation}
To prove this, for all $m\in\N$, define the open sets
\begin{equation*}
    \cX_m:={\rm span}\{\psi_1,\dots,\psi_m\}+B_{r}^{\cX},
\end{equation*}
  where $B_{r}^{\cX}$ denotes the open ball
  of radius $r:=\frac{\delta^2}{\Lambda_\cX^2}>0$ around $0\in\cX$. 
  These sets are nested and
  $\bigcup_{m\in\N}\cX_m = \cX\supset K$. 
  Since $K$ is compact
  there exists ${N_1}$ such that $K\subseteq \cX_{{N_1}}$. 
  Then for any $x\in K$
  \begin{equation*}
    x=\sum_{j=1}^{{N_1}}\alpha_j\psi_j+\widetilde x
  \end{equation*}
  for some $\alpha_1,\dots,\alpha_{{N_1}}\in\R$ and some
  $\widetilde x\in\cX$ with
  $\norm{\widetilde x}_\cX^2<\frac{\delta^2}{\Lambda_\cX^2}$ and thus
  $\sum_{j\in\N}\inp{\widetilde
    x}{\widetilde\psi_j}_\cX^2<\frac{\delta^2}{\Lambda_\cX}$. Then for
  all $m\ge {N_1}$
  \begin{equation*}
    \norm{x-\sum_{j=1}^{m}\inp{x}{\widetilde\psi_j}_\cX\psi_j}_\cX
    =\norm{\sum_{j>m}\inp{\widetilde x}{\widetilde\psi_j}_\cX\psi_j}_\cX
    <
    \Lambda_\cX \sum_{j > N_1}\inp{\widetilde x}{\widetilde\psi_j}_\cX^2\le\delta^2,
  \end{equation*}
  which shows \eqref{eq:claim_m1}.

  {\bf Step 2.} 
  We claim that for every
  $\eps>0$ exists ${N_2}(\eps,K)\in\N$ such that
  \begin{equation}\label{eq:claim_m2}
    \sup_{x\in K}\norm{\cG(x)-\cG\left(\sum_{j=1}^m\inp{x}{\widetilde\psi_j}_\cX\psi_j\right)}_\cY
    \le
    \eps\qquad\qquad\forall m\ge N_2.
  \end{equation}

  Compactness of $K$ implies that $\cG:K\to\cY$ is uniformly continuous.  
  Hence for any $\eps>0$ exists $\delta>0$ such that
  $\norm{x-\widetilde x}_\cX<\delta$ implies
  $\norm{\cG(x)-\cG(\widetilde x)}_\cY < \eps$. 
  Set
  ${N_2}(\eps,K) := {N_1}(\delta,K)$. 
  Then \eqref{eq:claim_m1} gives \eqref{eq:claim_m2}.

  {\bf Step 3.} 
  We claim that for every $\eps>0$ 
    there exists $M_3(\eps,K)\in\N$ such that\footnote{The notation 
    $N$ and $M$ is chosen so that $N$ always refers to trunction of
    input, $M$ always refers to truncation of output. 
    The subindex of $N$ and $M$ refers to step of proof.}
  \begin{equation}\label{eq:claim_m4}
    \sup_{x\in K}\sup_{n\in\N} 
    \sum_{j> M_3}\inpc{\cG\left(\sum_{i=1}^n\inp{x}{\widetilde\psi_i}_\cX\psi_i\right)}{\widetilde\eta_j}_\cY^2
    < \eps^2 
     \;.
  \end{equation}

As $x\mapsto \cE_\cX(x):\cX\to\ell^2(\N)$ is a bounded and
 linear (thus continuous) map and $K\subseteq\cX$ is compact,
 also $\cE_\cX(K)$ is compact. 
 According to Lemma \ref{lem:S} the set
  \begin{equation*}
    S:=\cE_\cX(K)\cup \{(c_1,\dots,c_n,0,0,\dots) : \bc\in\cE_\cX(K),~n\in\N\}\subseteq \ell^2(\N)\\
  \end{equation*}
  is compact.
Continuity of
$\cG\circ\cD_\cX:\ell^2(\N)\to\cY:\bc\mapsto
\cG(\sum_{j\in\N}c_j\psi_j)$ gives that
  \begin{equation}
    \cG(\cD_\cX(S)) 
    = 
    \bigg\{\cG(x):x\in K\bigg\}\cup \bigg\{\cG\bigg(\sum_{j=1}^n\inp{x}{\widetilde\psi_j}_\cX\psi_j\bigg):x\in K,~n\in\N\bigg\}
    \subseteq \cY
  \end{equation}
    is compact, and
  \eqref{eq:claim_m4} then follows by the statement in Step 1.
  
 {\bf Step 4.} We construct $\widetilde\cG_n$ and conclude the proof.

  For every $n\in\N$
  and $j\in\{1,\dots,n\}$
  let $\widetilde{\mathfrak{g}}_j^n:[-n,n]^n\to\R$ be a $\sigma$-NN
  such that
  \begin{equation}\label{eq:frgjn}
\sup_{\bc\in [-n,n]^n}
\bigg|\underbrace{\inpc{\cG\Big(\sum_{i=1}^nc_i\psi_i\Big)}{\widetilde\eta_j}_\cY}_{=:{\mathfrak{g}_j^n}(\bc)}-\widetilde{\mathfrak{g}}_j^n(\bc)\bigg|<2^{-n}.
  \end{equation}
  Such $\widetilde{\mathfrak{g}}_j^n$
  exists according to \cite[Theorem 1]{LESHNO1993861}, 
  since $\bc\mapsto g_j^n(\bc):\R^n\to\R$
  is continuous due to the continuity of $\cG$. 
  Now define (independent of $K$)
  \begin{equation*}
    \widetilde\cG_n(x)
    :=
    \sum_{j=1}^n\widetilde{\mathfrak{g}}_j^n(\inp{x}{\widetilde\psi_1}_\cX,\dots,\inp{x}{\widetilde\psi_n}_\cX)\eta_j.
  \end{equation*}

  Fix 
  $\eps>0$ and let 
  $n \ge \max\{N_2(\eps,K),M_3(\eps,K)\}$. 
  Then for any $x\in K$ by \eqref{eq:frgjn} and \eqref{eq:claim_m4}
  \begin{align*}
    &\norm{\cG\left(\sum_{j=1}^n\inp{x}{\widetilde\psi_j}_\cX\psi_j\right) - \widetilde \cG_n(x)}_\cY^2\\
    &=\norm{\sum_{j=1}^n (\mathfrak{g}_j^n
      ((\inp{x}{\widetilde\psi_j}_\cX)_{j=1}^n)- \widetilde{\mathfrak{g}}_j^n ((\inp{x}{\widetilde\psi_j}_\cX)_{j=1}^n))\eta_j
      + \sum_{j>n} \mathfrak{g}_j^n((\inp{x}{\widetilde\psi_j}_\cX)_{j=1}^n) \eta_j
      }_\cY^2\nonumber\\
    &\le 2\Lambda_\cY \left(\sum_{j=1}^n 2^{-2n}+\sum_{j>n}\mathfrak{g}_j^n((\inp{x}{\widetilde\psi_j}_\cX)_{j=1}^n)^2\right)
        \nonumber\\
    &\le 2\Lambda_\cY (n 2^{-2n}+\eps^2).
  \end{align*}
  Therefore using \eqref{eq:claim_m2}
  \begin{equation*}
    \lim_{n\to\infty}\sup_{x\in K}\norm{\cG(x)-\widetilde \cG_n(x)}_\cY
    \le
    \lim_{n\to\infty}
    \big(\eps+(2\Lambda_\cY (n2^{-2n}+\eps^2))^{1/2}\big)
    = (1+\sqrt{2\Lambda_\cY}) \eps.
  \end{equation*}
  Since $\eps>0$ was arbitrary, this concludes the proof.
\end{proof}
{\small 
\addcontentsline{toc}{section}{References}
\bibliographystyle{abbrv}
\bibliography{references}
}
\end{document}